\documentclass{amsart}

\usepackage[T1]{fontenc}
\usepackage{times}
\usepackage{amssymb,verbatim}
\theoremstyle{plain}
\usepackage[T1]{fontenc}
\usepackage[utf8]{inputenc}
\usepackage[english]{babel}
\usepackage{amsthm}
\usepackage{amsmath}
\usepackage{amstext}
\usepackage{float}
\usepackage{hyperref}
\usepackage{tikz}
\usetikzlibrary{matrix}

\newtheorem{thm}{Theorem}[section]
\newtheorem*{conj*}{Conjecture}
\newtheorem{conj}{Conjecture}[section]

\newtheorem{lemma}[thm]{Lemma}
\newtheorem{prop}[thm]{Proposition}
\newtheorem{cor}[thm]{Corollary}
\newtheorem{problem}{Problem}

\newtheorem{THM}{Theorem}

\theoremstyle{remark}

\newtheorem{remark}[thm]{Remark}

\newcommand{\Fol}[2][3]{\textsf{Fol}_{#2}(\mathbb P^{#1})}
\newcommand{\Open}[2][3]{\textsf{U}_{#2}(\mathbb P^{#1})}
\newcommand{\Log}{\textsf{Log}}
\newcommand{\LPB}{\textsf{LPB}}
\newcommand{\TM}{\textsf{TM}}
\newcommand{\TA}{\textsf{TA}}
\newcommand{\SLog}{\textsf{SLog}}
\newcommand{\mb}{\mathbb}
\newcommand{\mc}{\mathcal}

\newcommand{\C}{\mb C}
\newcommand{\Pj}{\mb P}

\newcommand{\N}{\mb N}

\newcommand{\Proj}{\mathbb{P}} %nuevo comando definido para projetivos
\newcommand{\Ol}{\mathcal{O}} %nuevo comando definido para holomorfo

\newcommand{\F}{\mc F}
\newcommand{\G}{\mc G}

\DeclareMathOperator{\Div}{Div}

\DeclareMathOperator{\Image}{Image}
\DeclareMathOperator{\codim}{codim}
\DeclareMathOperator{\sing}{sing}
\DeclareMathOperator{\rank}{rank}

\DeclareMathOperator{\Spec}{Spec}

\DeclareMathOperator{\Res}{Res}
\DeclareMathOperator{\Aut}{Aut}
\DeclareMathOperator{\Mor}{Mor}

\newcommand{\aut}{ \mathfrak{aut}}

\newcommand{\fix}{ \mathfrak{fix}}
\newcommand{\TF}{{T_\F}}
\newcommand{\NF}{{N_\F}}

\numberwithin{equation}{section}

\addtocounter{section}{0}             % Start with section 1
\numberwithin{equation}{section}       % Number formulas within sections
\title{Codimension one foliations of degree three on projective spaces}

\author[R.C.~da~Costa]{Raphael Constant da Costa}
\address{UERJ, Universidade do Estado do Rio de Janeiro, Rua São Francisco ˜Xavier, 524, Maracanã, 20550-900, Rio de Janeiro, Brazil. }
\email{raphaelconstant@ime.uerj.br}

\author[R.~Lizarbe]{Ruben Lizarbe}
\address{UERJ, Universidade do Estado do Rio de Janeiro, Rua São Francisco ˜Xavier, 524, Maracanã, 20550-900, Rio de Janeiro, Brazil. }
\email{ruben.monje@ime.uerj.br}

\author[J.V.~Pereira]{Jorge Vit\'orio Pereira}
\address{IMPA, Estrada Dona Castorina, 110, Jardim Botânico, 22460-320, Rio de Janeiro, Brazil.}
\email{jvp@impa.br}


\begin{document}

\begin{abstract}
We establish a structure theorem for degree three codimension one foliations on projective spaces
of dimension $n\ge 3$, extending a result by Loray, Pereira, and Touzet for degree three foliations
on $\mathbb P^3$. We show that the space of codimension one foliations of degree three on $\mathbb{P}^n$, $n\ge 3$,
has exactly $18$ distinct irreducible components parameterizing foliations without rational first integrals, and
at least $6$ distinct irreducible components parameterizing foliations with rational first integrals.
\end{abstract}

\maketitle

\setcounter{tocdepth}{1}

\tableofcontents

\section{Introduction}

In this work we study the decomposition of the space of degree three codimension one foliations on the projective space $\mathbb P^n$ into irreducible components.

\subsection{Irreducible components in degree zero, one, and two}
The space of degree zero codimension one foliations  on $\mathbb P^n$   is an irreducible smooth subvariety of $\mathbb PH^0(\mathbb P^n, \Omega^1_{\mathbb P^n}(2))$ isomorphic to the Grassmannian of lines on $\mathbb P^n$. It is unclear where this fact was first established. For a recent proof, see \cite{CerveauDeserti}.

For $n\ge 3$, the space of codimension one foliations of degree one on $\mathbb P^n$ is no longer irreducible. Jouanolou proved in \cite{Jouanolou} that they have  two irreducible components.

Almost two decades after the appearance of \cite{Jouanolou}, Cerveau and Lins Neto treated the next case and proved
in \cite{CerveauLins96} that for every $n\ge 3$, the space of codimension foliations of degree two on $\mathbb P^n$ has six irreducible components. This work spurred a lot of activity on the
study of irreducible components of the space of  foliations both in
codimension one and, more recently, in higher codimensions. Nevertheless, the focus moved
from the enumeration of the irreducible components of foliations of a certain
degree to the detection of families of irreducible components admitting
a uniform description, see for instance  \cite{Calvo,CaCeGiLi,CukiermanPereira,CukiermanPereiraVainsencher,MR3623760,Lizarbe17,CukiermanAceaMassri,Constant} and references therein.

\subsection{Structure of foliations of degree three}
Concerning degree three foliations on $\mathbb P^n$, $n\ge 3$,  Cerveau and Lins Neto established in \cite{MR3088436}
a rough structure theorem for them. They show that degree three foliations are either  rational pull-backs of foliations on projective surfaces or are  transversely affine (i.e., admit Liouvillian first integrals). Their methods do not provide control on the rational maps from projective spaces to surfaces, nor on the structure of the polar divisor of the transversely affine structure. Recently, the third author of this paper, together with Loray and Touzet, provided a more precise structure theorem for degree three foliations on $\mathbb P^3$, cf.  \cite[Theorem A]{LorayPereiraTouzet17}.
Our first result is an extension of this result to projective spaces of arbitrary dimension.

\begin{THM}\label{T:structure}
If  $\F$ is a codimension one singular holomorphic foliation on $\mathbb P^n$, $n\ge 3$, of degree three then
\begin{enumerate}
        \item $\F$ is defined by a closed rational $1$-form without codimension one zeros; or
        \item there exists an algebraically integrable codimension two foliation of degree one tangent to $\F$; or
        \item $\F$ is a linear pull-back of a degree $3$ foliation on $\mathbb P^2$; or
        \item $\F$ admits a rational first integral.
    \end{enumerate}
\end{THM}

The proof of Theorem \ref{T:structure}  is an extension of the arguments used to establish \cite[Theorem A]{LorayPereiraTouzet17}, see
Subsection \ref{SS:proof of structure}.

\subsection{Irreducible components}
We explore the classification provided by Theorem \ref{T:structure}
to  exhibit a complete list of the irreducible components of the space of foliations of $\mathbb P^n$, $n\ge3$,
whose general elements do not admit a rational first integral.

\begin{THM}\label{THM:A}
    The space of codimension one foliations of degree three on $\mathbb P^n$, $n\ge 3$, has exactly
    $18$ distinct irreducible components whose general elements correspond to foliations which
    do not admit a rational first integral.
\end{THM}

We actually prove a more precise result, see Theorem \ref{T:more precise}. Our methods provide an explicit description of each
of these $18$ irreducible components.  By computing the Zariski tangent space of the space of foliations on $\mathbb P^3$ at the general element of each one of these components, we could verify that among the $18$ irreducible components given by Theorem \ref{THM:A},
$12$ are generically reduced while $6$ are generically non-reduced. These are the first examples of generically non-reduced irreducible components of the space of foliations on projective spaces.

We were not able to give a complete classification of the irreducible components of the space of foliations for which the general elements correspond to algebraically integrable foliations, but we were able to prove the following partial result.

\begin{THM}\label{THM:B}
    The space of codimension one foliations of degree three on $\mathbb P^n$, $n\ge 3$, has at least $24$ distinct irreducible components.
\end{THM}

In other words, there are at least six irreducible components whose general elements correspond to algebraically integrable foliations.
Notice that Theorem \ref{THM:B} does not claim that the number of irreducible components of the space of degree three foliations on $\mathbb P^n$, $n\ge 3$, does not vary with the dimension $n$.

\subsection{Organization of the paper}
Section \ref{S:space} is introductory and presents the basic definitions of our subject. Section \ref{S:LOG} reviews known result
about the so-called logarithmic components and relate them to set of foliations admitting polynomial integrating factors. It also presents
two previously unknown irreducible components of the space of degree  three foliations on $\mathbb P^3$. Section \ref{S:mult} investigates degree three foliations on $\mathbb P^3$ tangent to actions of $\mathbb C^*$, and Section \ref{S:additive} investigates degree three foliations on $\mathbb P^3$ tangent to algebraic actions of $\mathbb C$. These two sections together constitute the technical core of the proofs of Theorem \ref{THM:A} and \ref{THM:B}.  In Section \ref{S:proofs} we prove  Theorem  \ref{T:structure}, present refined versions of Theorems \ref{THM:A} and \ref{THM:B}, and prove them. This section, and the paper, finish with some open problems and one conjecture.

\subsection{Acknowledgments}
R.C.C.~and~R.L.~were supported by FAPERJ. J.V.P.~was supported by CNPq, and FAPERJ. The computations in this paper were performed by using Maple, Version 18.02.

\section{The space of foliations on projective spaces}\label{S:space}

\subsection{Twisted differentials defining foliations}
A codimension one foliation $\mathcal F$ on a complex manifold $X$ is defined
by a section $\omega$ of  $\Omega^1_{X}\otimes \NF$, where $\NF$ is a line bundle called the normal bundle of $\F$,
which has no divisorial components in its zero set and satisfies Frobenius integrability condition
\[
    \omega \wedge d \omega =0 \, .
\]
Two $1$-forms satisfying these conditions define the same foliation if, and only if, they differ multiplicatively by a nowhere vanishing holomorphic function. In particular, if $X$ is compact, then they must differ by a non-zero constant. Note that $d\omega$ does not have an intrinsic meaning, but $\omega\wedge d\omega$ is an unambiguously defined section of $\Omega^3_X \otimes \NF^{\otimes 2}$.

When $X$ is the projective space $\mathbb P^n$, $n \ge 2$, a $1$-form $\omega \in H^0(\mathbb P^n, \Omega^1_{\mathbb P^n}\otimes \NF)$ acquires a rather concrete flavor. If we define the degree $d$ of $\F$ as the number of tangencies of $\F$ with a general line $i: \mathbb P^1 \to \mathbb P^n$, then $d$ is the degree of the zero locus of $i^*\omega \in H^0(\mathbb P^1, \Omega^1_{\mathbb P^1} \otimes i^* \NF)$. Therefore $\NF = \mathcal O_{\mathbb P^n}(d+2)$. We can use the following twist of Euler's sequence,
\[
    0 \to \Omega^1_{\mathbb P^n}(d+2) \to   \mathcal O_{\mathbb P^n}(d+1)^{\oplus n+1} \to \mathcal O_{\mathbb P^n}(d+2) \to 0 \, ,
\]
to identify $H^0(\mathbb P^n, \Omega^1_{\mathbb P^n}(d+2))$ with the $\mathbb C$-vector space of homogenous
polynomial $1$-forms
\[
     \sum_{i=0}^n a_i(x_0, \ldots, x_n) dx_i
\]
with coefficients $a_0, \ldots, a_n$ of degree $d+1$,  which are annihilated by
the radial vector field
$R = \sum_{i=0}^n x_i \frac{\partial}{\partial x_i}$.

Frobenius integrability condition becomes
\[
     \left( \sum_{i=0}^n a_i(x_0, \ldots, x_n) dx_i \right) \wedge \left( \sum_{i=0}^n \sum_{j=0}^n \frac{\partial a_i}{\partial x_j}(x_0, \ldots, x_n) dx_j \wedge dx_i  \right) = 0 \, .
\]

\subsection{Space of foliations on projective spaces}\label{SS:Space}
The space of codimension one foliations of degree $d$ on the projective space $\mathbb P^n$ is, by definition, the
locally closed subscheme $\Fol[n]{d}$ of $\mathbb P H^0(\mathbb P^n, \Omega^1_{\mathbb P^n}(d+2))$
defined by the conditions
\[
    [\omega] \in \Fol[n]{d} \text{ if, and only if, } \codim \sing(\omega) \ge 2 \text{ and }  \omega \wedge d \omega = 0\,.
\]

We will denote by $\Open[n]{d}$ the open subscheme of $\mathbb P H^0(\mathbb P^n, \Omega^1_{\mathbb P^n}(d+2))$
defined by the open condition $\codim \sing(\omega) \ge 2$.

For any $n\ge 3$ and any $d\ge 1$, $\Fol[n]{d}$ is not irreducible. Nevertheless,
$\Fol[n]{d}$ is always connected as shown by the result below, see also \cite[Theorem 2.5]{Soares}.

\begin{prop}
    For every $d\ge 0$ and every $n\ge 2$, the space $\Fol[n]{d}$ of foliations of degree $d$
    on $\mathbb P^n$ is connected.
\end{prop}
\begin{proof}
    If $n=2$ then Frobenius integrability condition is automatically satisfied. Therefore  $\Fol[2]{d}$, for every $d \ge 0$,
    is a  Zariski open subset of $\mathbb P H^0(\mathbb P^2, \Omega^1_{\mathbb P^2}(d+2))$. In particular, $\Fol[2]{d}$ is connected
    for every $d \ge 0$.

    Assume now that $n \ge 3$ and let $\F$ and $\F'$ be two distinct foliations on $\mathbb P^n$ of the same degree $d$.
    Let $[\omega]$ and $[\omega']$ be  elements of $\mathbb P H^0(\mathbb P^n,\Omega^1_{\mathbb P^n}(d+2))$ defining, respectively, $\F$ and $\F'$.
    Let $i: \mathbb P^2 \to \mathbb P^n$  be a immersion of $\mathbb P^2$ in $\mathbb P^n$ such that both
    $i^* \omega$ and $i^* \omega'$ are twisted differentials with isolated zeros. A general linear immersion has such property. Consider now
    a holomorphic family of automorphism $\varphi_t$ of $\mathbb P^n$ parameterized by $\mathbb C^*$ such that   $\varphi_t$ is the identity on $i(\mathbb P^2)$, $\lim_{t \to 0} \varphi_t$ is a linear projection $\mathbb P^n \dashrightarrow \mathbb P^2$ with center away from $\mathbb P^2$, and $\varphi_1$ is the identity of $\mathbb P^n$. If one choose homogeneous
    coordinates $(x_0:x_1:x_2: \ldots : x_n)$ such that $i(\mathbb P^2) = \{ x_3 = \ldots = x_n =0\}$ then it suffices to take
    \[
        \varphi_t(x_0: \ldots : x_n) = (x_0 : x_1 : x_2 : t x_3 :t x_4 : \ldots : tx_n) \, .
    \]

    Due to the connectedness of $\Fol[2]{d}$, $i^*[\omega]$ and $i^*[\omega']$ belong to the same connected component. It follows that the same holds true for $\varphi_0^*[\omega]= \varphi_0^* (i^* [\omega])$ and $\varphi_0^*[\omega']=\varphi_0^* (i^* [\omega'])$.
	Finally use $\varphi_t$ to connect $[\omega]$ to $\varphi_0^*[\omega]$, and do the same to connect $[\omega']$ to $\varphi_0^*[\omega']$, obtaining the result.
\end{proof}

\subsection{Zariski tangent space of $\Fol[n]{d}$}
Let $[\omega] \in \Fol[n]{d}$ be a foliation. A first-order integrable deformation of $[\omega]$ is
a morphism
\[
    \gamma : \Spec \left( \frac{\mathbb C[\varepsilon]}{\varepsilon^2} \right) \to \Fol[n]{d}
\]
which maps the unique closed point on the left to $[\omega]$. Concretely, a first-order deformation of $[\omega]$
is given by a section $\omega_1 \in H^0(\mathbb P^n, \Omega^1_{\mathbb P^n}(d+2))$ such that
\[
   ( \omega + \varepsilon \omega_1 ) \wedge d( \omega + \varepsilon \omega_1 ) =0 \mod \varepsilon^2  \, .
\]
Moreover, two sections $\omega_1$ \and $\omega_1'$ define the same first-order deformation $\gamma$ if, and only
if, they differ by a constant multiple of $\omega$. In other words, the Zariski tangent space  of $\Fol[n]{d}$ at $[\omega]$ is
equal to the vector space
\[
    T_{\Fol[n]{d}}([\omega]) = \frac{ \left\{ \omega_1 \in H^0(\mathbb P^n, \Omega^1_{\mathbb P^n}(d+2)) \, | \, \omega\wedge d\omega_1 + \omega_1 \wedge d\omega =0 \right\} }{\mathbb C \omega } \, .
\]

The lemma below provides a computationally verifiable criterion for a closed subvariety of $\Fol[n]{d}$ be an irreducible
component of $\Fol[n]{d}$. Its proof is standard, and we register the statement here for later reference.

\begin{lemma}\label{L:genred}
    Let $\Sigma \subset \Fol[n]{d}$ be an irreducible subvariety. If there exists $[\omega] \in \Sigma$ such that the
    Zariski tangent space of $\Fol[n]{d}$ has dimension equal to the dimension of $\Sigma$ then $\Sigma$ is an
    irreducible component of $\Fol[n]{d}$. Moreover, the scheme structure of $\Fol[n]{d}$ at a general point of
    $\Sigma$ is reduced.
\end{lemma}

An irreducible component $\Sigma$ of $\Fol[n]{d}$ satisfying the hypothesis of Lemma \ref{L:genred} will
be called a generically reduced component of $\Fol[n]{d}$.

\section{Foliations defined by logarithmic \texorpdfstring{$1$-forms}{1-forms}}\label{S:LOG}

\subsection{Logarithmic $1$-forms}
Given a reduced divisor $D$ on a complex manifold $X$, a logarithmic $1$-form with poles
on $D$ is a meromorphic $1$-form $\omega$ which can be written as
\begin{equation}\label{E:loglocal}
    \omega = \alpha + \sum_{i=1}^k g_i \frac{df_i}{f_i}
\end{equation}
at sufficiently small neighborhoods of arbitrary points of $X$, where $\{f_1\cdots f_k = 0\}$ is a reduced
local  equation for the support of  $D$, $\alpha$ is a holomorphic $1$-form, and $g_1, \ldots, g_k$ are holomorphic functions.
The sheaf of logarithmic $1$-forms with poles along $D$ is denoted by $\Omega^1_X(\log D)$.

\begin{remark}\label{R:explica log forms}
    The definition of $\Omega^1_X(\log D)$ presented above does not agree with Saito's  definition \cite{Saito}.
    According to Saito, a meromorphic $1$-form $\omega$ is logarithmic with poles on $D$
    if both $\omega$ and $d\omega$ have simple poles, and the poles are contained in the support of $D$.
    Both definitions agree if $D$ is a normal crossing divisor. Moreover, the definition above behaves
    well with respect to proper morphisms. In particular, if $\pi:Y \to X$ is a proper morphism
    such that $\pi^* D$ is supported on a simple normal crossing divisor and $\omega$ is $1$-form on $X$ with logarithmic poles
    on $D$ then $\pi^* \omega$ is a logarithmic $1$-form with poles on $\pi^*D$.
\end{remark}

If we write $D$ as the locally finite sum
$\sum_i D_i$ of the irreducible components, then $\Omega^1_X(\log D)$ fits into the exact sequence
\[
    0 \to \Omega^1_X \to \Omega^1_X(\log D) \to \bigoplus_i \mathcal O_{\tilde D_i} \to 0 \, ,
\]
where $\tilde D_i$ is the normalization of $D_i$ and the morphism on the right is the residue morphism
which sends $\omega$ as in (\ref{E:loglocal}) to the direct sum of  the pull-backs of $g_i$ to the normalization of $D_i$.

When the irreducible components of $D$ are compact, we can represent the image of $\omega \in H^0(X,\Omega^1_X(\log D))$ under the residue morphism as a $\mathbb C$-divisor $\Res(\omega) = \sum_i \lambda_i D_i$.

Assume now that $X$ is projective, or at least K\"ahler compact. We recall that logarithmic $1$-forms
on compact K\"ahler manifolds with poles on  normal crossing divisors are closed according to \cite[Corollary 3.2.14]{Deligne}.
Remark \ref{R:explica log forms} implies that the same holds  for our definition of logarithmic $1$-forms, independently of the
kind of singularities of $D$.

\begin{prop}\label{P:realizalog}
    Let $R$ be a $\mathbb C$-divisor with support contained in a reduced divisor $D$. There exists
    $\omega \in H^0(X,\Omega^1_X(\log D))$ with $\Res(\omega)=R$ if, and only if, $R$ is mapped to
    zero under the Chern class morphism $\Div(X) \otimes \mathbb C \to H^2(X,\mathbb C)$. Moreover,
    two $1$-forms $\omega, \omega' \in H^0(X,\Omega^1_X(\log D))$ have the same residue divisor
    if and only if their difference is a holomorphic $1$-form.
\end{prop}
\begin{proof}
    The result is well-known. It is a particular case of \cite[Proposition 2.2]{CousinPereira}.
\end{proof}

Again, when $X = \mathbb P^n$, logarithmic $1$-forms can be explicitly written with ease.
A reduced divisor $D = D_1 + \ldots + D_k$ on $\mathbb P^n$,  supported on $k$ irreducible hypersurfaces, is defined  as the zero locus
of $\{ f_1 \cdots f_k =0 \}$ where $f_i$ are irreducible homogeneous polynomials on $\mathbb C^{n+1}$. A logarithmic $1$-form $\omega$ with poles
on $D$ can be written as
\[
    \omega= \sum_{i=1}^k \lambda_i \frac{df_i}{f_i}
\]
where $\lambda_1, \ldots, \lambda_k$ are complex constants subjected to the condition
\[
    \sum_{i=1}^k \lambda_i \deg(f_i) = 0 \, .
\]
This last equality is equivalent to the condition on the Chern class of the residue divisor
given by Proposition \ref{P:realizalog}.

\subsection{Logarithmic components}\label{SS:Log}
Given a logarithmic $1$-form $\omega \in H^0(X, \Omega^1_X(\log D))$, one can naturally produce a
twisted  differential $1$-form $\omega'$ in $H^0(X, \Omega^1_X \otimes \mathcal O_X(D))$ by multiplying $\omega$
by a section of $\mathcal O_X(D)$ which vanishes on $D$. If the zero locus of $\omega$ (as a section
of $\Omega^1_X(\log D)$) has codimension at least two and the residue divisor of $\omega$ has support equal to  $D$, then the twisted $1$-form $\omega'$ also has codimension two zeros. If moreover $\omega$ is closed, what is always the case if $X$ is projective, then $\omega'$ is integrable and defines a codimension one foliation $\F$ on $X$.

Given a reduced divisor $D = D_1 + \ldots + D_k$ on a projective manifold $X$, we can vary $\omega \in H^0(X, \Omega_X(\log D))$ and vary $D_1, \ldots, D_k$ in their linear equivalence classes, in order to produce a subvariety of $\mathbb P H^0(X, \Omega^1_X \otimes \mathcal O_X(D))$ formed by integrable twisted $1$-forms.

Suppose now that $X = \mathbb P^n$. In order to describe explicitly the resulting subvariety  and fix notation for what follows, let $(d_1, \ldots, d_k)$ be an ordered  partition of $d+2$, i.e. $1 \le d_1 \le d_2 \le \ldots \le d_k$ and $\sum_{i=1}^k d_i = d+2$.
Set
\[
    \Lambda = \left\{ ( \lambda_1, \ldots, \lambda_k ) \in \mathbb C^k | \sum_{i=1}^k \lambda_i d_i = 0 \right\}
\]
and consider the rational map
\begin{align*}
    \Phi_{n,d_1, \ldots, d_k} : \mathbb P(\Lambda) \times \left( \prod_{i=1}^k \mathbb P H^0(\mathbb P^n,\mathcal O_{\mathbb P^n}(d_i)) \right) &\dashrightarrow \mathbb P H^0(\mathbb P^n, \Omega^1_{\mathbb P^n}(d+2))  \\
     ( [\lambda_1: \ldots: \lambda_k] , [f_1],\ldots, [f_k] ) &\mapsto \left( \prod_{i=1}^k f_i \right) \left( \sum_{i=1}^k \lambda_i \frac{df_i}{f_i} \right) \, .
\end{align*}
The logarithmic components of $\Fol[n]{d}$ are the subvarieties
\[
   \Log(\mathbb P^n)(d_1, \ldots, d_k) = \overline{\Image(\Phi_{n,d_1, \ldots, d_k})} \cap \Open[n]{d}
\]
indexed by ordered partitions $(d_1, \ldots, d_k)$ of $d+2$ with at least two terms. Recall from Subsection \ref{SS:Space} that
$\Open[n]{d}$ is the open subscheme of $\mathbb P H^0(\mathbb P^n, \Omega^1_{\mathbb P^n}(d+2))$
defined by the open condition $\codim \sing(\omega) \ge 2$. Some authors call the
logarithmic components associated to partitions with only two terms, rational components as any foliation parameterized by them admit rational first integrals.

\begin{thm}\label{T:log classifica}
    If $n \ge 3$ and $(d_1, \ldots, d_k)$ is an ordered partition of $d+2$ with at least two terms then
    $\Log(\mathbb P^n)(d_1, \ldots, d_k)$ is a generically reduced  irreducible component of $\Fol[n]{d}$.
\end{thm}
\begin{proof}
    We will not present an actual proof, but instead provide pointers to the literature where the different parts of the statement above were first established.

    For ordered partitions with exactly two terms, the fact that $\Log(\mathbb P^n)(d_1,d_2)$ is an irreducible component of $\Fol[n]{d}$  follows from \cite{GomezLins}. Later, an alternative proof of this result using infinitesimal methods appeared in  \cite{CukiermanPereiraVainsencher}, which also implied that these irreducible components are generically reduced.

    For ordered partitions with three terms or more, the fact that  $\Log(\mathbb P^n)(d_1,\ldots, d_k)$ is an irreducible component of $\Fol[n]{d}$ is  proved in \cite{Calvo}. A proof using infinitesimal methods and establishing that these components are generically reduced, appeared only very recently in \cite{CukiermanAceaMassri}. The arguments are substantially more involved than the ones used in the case of ordered partitions with two terms.
\end{proof}

\subsection{Logaritmic components and polynomial integrating factors}
Fix $d\ge 0$ and let $\omega \in H^0(\mathbb P^n, \Omega^1_{\mathbb P^n}(d+2))$ be a twisted differential.
A non-zero section $P \in H^0(\mathbb P^n, \mathcal O_{\mathbb P^n}(d+2))$ is, by definition, a polynomial integrating factor for $\omega$ if the rational $1$-form $P^{-1} \omega$ is closed, i.e.
\[
    d \left( \frac{\omega}{P} \right) =0  \, .
\]
The existence of a polynomial integrating factor  implies the integrability of $\omega$. 
Beware that in the definition of polynomial integrating factor for $\omega$, we are not
assuming that $\omega$ has zeros of codimension at least two, it can have divisorial components
in its zero set.

\begin{remark}
    If $\omega \in H^0(\mathbb P^n, \Omega^1_{\mathbb P^n}(d+2))$ is a $1$-form without  codimension one zeros defining a foliation $\mathcal F$ then  the existence of a polynomial integrating factor for $\omega$  is equivalent to the existence of a closed rational $1$-form without codimension one zeros defining the foliation $\mathcal F$.
\end{remark}

The relevance of the concept of polynomial integrating factor for our study
is made transparent by our next three results.

\begin{prop}\label{P:Key argument}
    Let $\F$ be a degree three codimension one foliation on $\mathbb P^3$.
    If the foliation $\F$ does not admit a rational first integral and
    does not admit a polynomial integrating factor then there exists an action
    of a one-dimensional algebraic group with orbits tangent to $\F$.
\end{prop}
\begin{proof}
    The proof relies on \cite[Theorem A]{LorayPereiraTouzet17}, which is nothing more
    than the specialization of Theorem \ref{T:structure} of this paper to $\mathbb P^3$.
    According to it, we have four possibilities for $\F$: 
    \begin{enumerate}
        \item \label{I:341} the foliation $\F$ is defined by a closed
    rational $1$-form without codimension one zeros; 
        \item \label{I:342} there exists a degree one  foliation by algebraic curves tangent to $\F$; 
        \item \label{I:343} the foliation $\F$ is a linear pull-back from $\mathbb P^2$;   or 
        \item \label{I:344} the foliation $\F$ admits a rational first integral.
    \end{enumerate}

    Our assumptions exclude \ref{I:341} and \ref{I:344}.  Therefore $\F$  is either a linear pull-back or tangent
    to a degree one foliation by algebraic curves. In both cases, there is an action of a one-dimensional
    algebraic group with orbits tangent to $\F$.
\end{proof}

\begin{lemma}\label{L:pifclosed}
The set of twisted  $1$-forms in $H^0(\mathbb P^n, \Omega^1_{\mathbb P^n}(d+2))$ which admits
a polynomial integrating factor is closed.
\end{lemma}
\begin{proof}
    For a fixed $\omega \in H^0(\mathbb P^n, \Omega^1_{\mathbb P^n}(d+2))$,  consider the  linear map
    \begin{align*}
        \delta_{\omega} : H^0(\mathbb P^n, \mathcal O_{\mathbb P^n}(d+2)) & \longrightarrow H^0(\mathbb P^n, \Omega^2_{\mathbb P^n}(2d+4)) \\
        P &\longmapsto dP \wedge \omega - P d \omega \, .
    \end{align*}
    The definition of $\delta_{\omega}$ implies that  $P \in H^0(\mathbb P^n, \mathcal O_{\mathbb P^n}(d+2))$ is a polynomial integrating factor
    for $\omega$ if, and only if, $\delta_{\omega}(P)=0$. In particular, $\omega$ admits a polynomial
    integrating factor, if and only if, $\ker \delta_{\omega} \neq 0$. The result follows from the lower
    semi-continuity of the function $\omega \mapsto \rank \delta_{\omega} \in \mathbb N$.
\end{proof}

\begin{prop}\label{P:belong to log}
    A foliation  $\F \in \Fol[n]{d}$ belongs to (at least) one of the logarithmic components if, and only if, it
    admits a polynomial integrating factor.
\end{prop}
\begin{proof}
    By definition, a general element of a logarithmic component admits a polynomial
    integrating factor.  Lemma \ref{L:pifclosed} implies that every  element in any given logarithmic component  also
    admits a polynomial integrating factor.

    Let $\omega \in H^0(\mathbb P^n, \Omega^1_{\mathbb P^n}(d+2))$ be a twisted differential defining the foliation $\F$, and let  $P \in H^0(\mathbb P^n, \mathcal O_{\mathbb P^n}(d+2))$ be a polynomial integrating factor of it.
    The quotient $P^{-1}\omega$ is a closed rational form. If one of the residues of $P^{-1}\omega$ is non-zero and at least one of the components of its polar divisor has order greater than one, then it follows from \cite[Lemma 8]{CerveauLins96} that $\F$ belongs to a logarithmic component whose the general element has poles on at least three distinct hypersurfaces.

    On the other hand, if $P^{-1}\omega$ has only simple poles, it is actually a logarithmic $1$-form, and if all the residues of $P^{-1}\omega$ are zero, $\F$ has a rational first integral and is defined by a logarithmic $1$-form without codimension one zeros.
    In the last two cases, it may happen that the logarithmic $1$-form defining $\F$ has only two distinct hypersurfaces in its polar set. In our terminology, the rational components of \cite{CerveauLins96} are also logarithmic components. Therefore, in any case, $\omega$ belongs to a logarithmic component.
\end{proof}

\subsection{Infinitesimal automorphisms and polynomial integrating factors}\label{SS:inf vs pif}
If $\F$ is a foliation on $\mathbb P^n$ defined by $\omega \in H^0(\mathbb P^n, \Omega^1_{\mathbb P^n}(d+2))$,  we will denote by $\mathfrak{aut}(\F)$ its Lie algebra of
infinitesimal automorphisms, i.e.
\[
    \mathfrak{aut}(\F) = \{ v \in H^0(\mathbb P^n, T_{\mathbb P^n}) \, | \, L_v \omega \wedge \omega =0 \} \,.
\]
The normal subalgebra formed by vector fields which annihilate $\omega$ will be denoted by $\mathfrak{fix}(\F)$, i.e.
\[
    \mathfrak{fix}(\F) = \{ v \in \mathfrak{aut}(\F) \, | \, i_v \omega =0 \}\, .
\]

The lemma below appeared  in \cite[Lemma 4.3]{LorayPereiraTouzet13} in a slightly different form.

\begin{lemma}\label{L:fix neq aut}
    Let $\F$ be a codimension one foliation on $\mathbb P^n$. If $\mathfrak{fix}(\F) \neq \mathfrak{aut}(\F)$ then
    $\F$ admits a polynomial integrating factor.
\end{lemma}
\begin{proof}
    Let $v \in \mathfrak{aut}(\F) - \mathfrak{fix}(\F)$. A simple computation shows that $(i_v \omega)$ belongs to $H^0(\mathbb P^n, \mathcal O_{\mathbb P^n}(d+2))$ is a polynomial integrating factor for $\omega$, see \cite[Corollary 2]{PereiraSanchez}.
\end{proof}

The Lie algebra $\aut(\F)$ is the Lie algebra of the linear algebraic group
\[
    \Aut(\F) = \{ \varphi \in \Aut(\mathbb P^n) \, | \, \varphi^* \omega \wedge \omega =0 \} \, .
\]
The complex Lie subgroup generated by (the exponential of) $\mathfrak{fix}(\F)$ is not necessarily closed.

\begin{cor}\label{C:fix not closed}
    Let $\F$ be a codimension one foliation on $\mathbb P^n$.
    If the subgroup generated by $\mathfrak{fix}(\F)$ is not a closed subgroup of $\Aut(\F)$ then
    $\F$ admits a polynomial integrating factor.
\end{cor}
\begin{proof}
    Since $\mathfrak{aut}(\F)$ is tangent to the closed group $\Aut(\F)$, the assumptions imply that $\mathfrak{aut}(\F)\neq \mathfrak{fix}(\F)$. The result follows from Lemma \ref{L:fix neq aut}.
\end{proof}

\subsection{Logarithmic $1$-forms with codimension one zeros}

In this subsection, we present two previously unknown irreducible components of $\Fol[3]{3}$. Both have general elements
defined by logarithmic $1$-forms with codimension one zeros.

\subsubsection{Special logarithmic component of type $\Log(3,4)$}
The logarithmic $1$-form
\begin{equation}\label{E:CeDe}
    \omega = d \log \frac{(x_0 x_4 ^3 - (2x_1 x_3 + x_2 ^2)x_4^2 + 2 x_2 x_3^2 x_4 - (1/2) x_3^4)^3}{(x_1x_4^2 - x_2x_3 x_4 + (1/3)x_3^3)^4}
\end{equation}
defined on $\mathbb P^4$ has zero divisor equal to $2H$, where $H$ is the hyperplane $\{x_4 =0\}$. Therefore $\omega$ defines a foliation $\F$ of degree $3$ on $\mathbb P^4$. The foliation $\F$ first appeared in \cite{CerveauDeserti} where it was found  in the course of the classification of foliations on $\mathbb P^4$ with trivial tangent sheaf, i.e. $\TF = \mathcal O_{\mathbb P^4}^{\oplus 3}$. There, it was conjectured that $\F$
is rigid,  i.e. the closure of the orbit of $\F$ under $\Aut(\mathbb P^4)$ is an irreducible component of $\Fol[4]{3}$. This conjecture was confirmed by \cite[Theorem 4]{CukiermanPereira}.

Let $M \subset \Mor(\mathbb P^3, \mathbb P^4)$ be the irreducible component of the space of morphisms
from $\mathbb P^3$ to $\mathbb P^4$ corresponding to morphisms $\varphi$ such that $\varphi^* \mathcal O_{\mathbb P^4}(1) = \mathcal O_{\mathbb P^3}(1)$. Consider the rational map
\begin{align*}
    \Phi : M &\dashrightarrow \mathbb P H^0(\mathbb P^3, \Omega^1_{\mathbb P^3}(5)) \\
    \varphi & \mapsto \varphi^* \F \, ,
\end{align*}
where $\F$ is the foliation on $\mathbb P^4$  described above.

\begin{prop}\label{P:Log34}
    The closure of the image of $\Phi$ intersected with $\Open[3]{3}$ is an irreducible and generically reduced component of $\Fol[3]{3}$
    of dimension $16$.
\end{prop}
\begin{proof}
    If $\varphi \in M$ is general then a computation shows that the Zariski tangent space of $\Fol[3]{3}$
    at $\varphi^* \F$ has dimension $16$ and that $\varphi^*\F$ has trivial automorphism group. Another computation
    shows that the differential of $\Phi$ at $\varphi$ has rank $16$. This is sufficient to show
    that the closure of the image of $\Phi$ is an irreducible and generically reduced component of $\Fol[3]{3}$
    of dimension $16$.
\end{proof}

We will denote by $\SLog(3,4)$ the irreducible component of $\Fol{3}$ described by Proposition \ref{P:Log34}.

\subsection{A note about computer-aided calculations}
In the proof of Proposition \ref{P:Log34} above  and  several other places of our paper, we carried out computations
using the software Maple. In all instances were we did not carry the computations also by hand, we used Maple
only to solve huge systems of linear equations, most notably to determine the Zariski tangent space of $\Fol{3}$ at a
given $1$-form. We have uploaded to arXiv, as ancillary files, the Maple code we used to perform all the computations in this paper.

\subsubsection{Special logarithmic component of type $\Log(2,5)$}

\begin{lemma}\label{L:conta25}
Let $g = x_0^2 - 2x_1 x_3 \in H^0(\mathbb P^3, \mathcal O_{\mathbb P^3}(2))$.
If $f \in H^0(\mathbb P^3, \mathcal O_{\mathbb P^3}(5))$ is such that $x_3^2$ divides
\[
    5 f dg - 2 g df
\]
then there exists constants $(c_1,c_2, \ldots, c_{11}) \in \mathbb C^{11}$ such that
\begin{align*}
 f  &=  c_1\left(-5{x_0}^{3}x_1x_3+{x_0}^{5}+\frac{15}{2}x_0{x_1}^{2}{x_3}^{2}\right)+c_2x_0{x_3}^{4}+  c_3x_0x_2{x_3}^{3}+c_4x_0x_1{x_3}^{3}+\\
    & + c_5{x_0}^{2}{x_3}^{3} + c_6{x_3}^{5}+c_7x_2{x_3}^{4} + c_8{x_2}^{2}{x_3}^{3}+c_9x_1{x_3}^{4}+c_{10}x_1x_2{x_3}^{3}+c_{11}{x_1}^{2}{x_3}^{3}
\end{align*}

\end{lemma}
\begin{proof}
    Straightforward, yet lengthy,  computation. One has to solve a  system of linear equations on the coefficients of the
    general element of $H^0(\mathbb P^3, \mathcal O_{\mathbb P^3}(5))$.
\end{proof}

\begin{prop}\label{P:special log(2,5)}
    Let $\Phi : \mathbb P ( \mathbb C^{11} ) \times \Aut(\mathbb P^3)  \dashrightarrow  \mathbb P H^0(\mathbb P^3, \Omega^1_{\mathbb P^3}(5))$ be  the rational map
    defined by
    \[
        (f, \varphi)  \mapsto \varphi^*\left(\left[\frac{5 f dg - 2 g df}{x_3^2}\right]\right) \, ,
    \]
    where $g = x_0^2 - 2x_1 x_3$ and $f$ is as in the conclusion of Lemma \ref{L:conta25}.
    Then the closure of the image of $\Phi$ intersected with $\Open{3}$ is a generically reduced irreducible component of
    $\Fol{3}$ of dimension $19$.
\end{prop}
\begin{proof}
    Let $G \subset \Aut(\mathbb P^3)$ be the subgroup preserving $\{g =0 \} \cup \{ x_3=0\}$.
    A computation shows that $G$ has dimension six. Therefore the dimension of the image of
    $\Phi$ is $10 + 15 -6 = 19$. At the same time, computing the dimension
    of the Zariski tangent space of $\Fol{3}$ at the image under $\Phi$ of a sufficiently general element
    of $\mathbb C^{11} \times \{\mathrm{id} \}$, one gets $19$.
    The proposition follows.
\end{proof}

We will denote by $\SLog(2,3)$ the irreducible component described by Proposition \ref{P:special log(2,5)}.

\tikzset{
    table/.style={
        matrix of nodes,
        row sep=-\pgflinewidth,
        column sep=-\pgflinewidth,
        nodes={
            rectangle,
            draw=black,
            align=center
        },
        minimum height=1.5em,
        text depth=0.5ex,
        text height=2ex,
        nodes in empty cells,
        every even row/.style={
            nodes={fill=gray!20}
        },
        column 1/.style={
            nodes={text width=3.7cm,align=left,font=\small}
        },
        column 2/.style={
            nodes={text width=0.8cm,align=center},font=\small},
        column 3/.style={
            nodes={text width=0.8cm,align=center},font=\small},
        column 4/.style={
            nodes={text width=3.4cm,align=left, font=\small}
        },
        row 1/.style={
            nodes={ font=\bfseries
            }
        }
    }
}
\begin{table}[H]
\begin{tikzpicture}
\matrix (mat) [table,text width=3cm]
{
    {\bf Irreducible Component} & ${\dim}$     &  $\mathrm{Zdim}$      & Comment   \\
    $\Log(1,1,1,1,1) $  & $18$     & $18$  &     \\
    $\Log(1,1,1,2)   $  & $20$     & $20$  &   \\
    $\Log(1,1,3)$       & $26$     & $26$  &   \\
    $\Log(1,2,2)$       & $22$     & $22$  &  \\
    $\Log(1,4) $        & $36$     & $36$  &   \\
    $\Log(2,3) $        & $28$     & $28$  &  \\
    $\SLog(2,5)$        & $19$     & $19$  &  Previously unknown.\\
    $\SLog(3,4)$        & $16$     & $16$  &  Previously unknown.\\
};
\end{tikzpicture}
\caption{\small The eight irreducible components of $\Fol{3}$ presented in Section \ref{S:LOG}}\label{Tab:LOG}
\end{table}

\subsection{Synthesis}
 We present in Table \ref{Tab:LOG} the irreducible components of $\Fol{3}$ discussed in Section \ref{S:LOG}. The second column indicates the dimension of the irreducible component and the third column indicates the dimension of the Zariski tangent of $\Fol{3}$ at the generic element of the corresponding set. In all examples discussed in this section, both dimensions agree. In the next section, we will present similar tables describing irreducible componentes whose generic elements correspond to foliations tangent to algebraic actions of $\mathbb C^*$.

\section{Foliations on \texorpdfstring{$\mathbb P^3$}{P3} tangent to multiplicative actions}\label{S:mult}

In this section we study foliations on $\mathbb P^3$ tangent to actions
of $\mathbb C^*$.

\subsection{Multiplicative actions on $\mathbb P^3$}
A non-trivial action of $\mathbb C^*$ on $\mathbb P^3$ can be presented,
in suitable homogeneous coordinates, as
\begin{align*}
    \varphi^{(a,b,c)} : \mathbb C^* \times \mathbb P^3 & \longrightarrow \mathbb P^3 \\
    (\lambda, [x_0:x_1:x_2:x_3]) & \mapsto [\lambda^a x_0 : \lambda^b x_1 : \lambda^c x_2 : x_3]
\end{align*}
where $0\le a \le b \le c \neq 0$ are non-negative integers. As we are interested in describing foliations invariant
by the action of $\varphi$, there is no loss of generality in assuming that the general orbit
of $\varphi$ has trivial isotropy group. This ammounts to assume that the greatest common divisor
of $a, b,$ and $c$ is equal to $1$.

In the same homogeneous coordinates,   write
\[
    v_{(a,b,c)} = a x_0 \frac{\partial}{\partial x_0} + b x_1 \frac{\partial}{\partial x_1} + c x_2 \frac{\partial}{\partial x_2} \, .
\]
The homogeneous vector field $v_{(a,b,c)}$  determines an element of  $H^0(\mathbb P^3, T_{\mathbb P^3})$  which generates $\varphi^{(a,b,c)}$.

For every $d \in \mathbb N$, and every $(a,b,c;n) \in \mathbb N^3\times \mathbb N$, set
\begin{equation*}
    V_d{(a,b,c;n)} = \left\lbrace \omega \in H^0(\mathbb P^3, \Omega^1_{\mathbb P^3}(d+2)) \, | \, i_{v_{(a,b,c)}} \omega = 0 \text{ and } L_{v_{(a,b,c)}} \omega = n \cdot \omega  \right\rbrace .
\end{equation*}
It follows from the definition of $V_d{(a,b,c;n)}$ that every $1$-form contained in it is integrable.

Define $\TM_d(a,b,c;n)$ as the closure of the image of the rational map
\begin{align*}
    \Psi  : \Aut(\mathbb P^3) \times \mathbb P(V_d(a,b,c;n)) &\dashrightarrow \mathbb P H^0(\mathbb P^3, \Omega^1_{\mathbb P^3}(d+2)) \\
    (\varphi, [\omega] ) & \longmapsto \varphi^* [\omega] \,
\end{align*}
intersected with $\Open[3]{d}$.

\begin{remark}\label{R:obvious}
    We will assume in several statements that $(a,b,c) \notin \{ (0,0,1), (1,1,1)\}$. The vector
    fields $v_{(0,0,1)}$ and $v_{(1,1,1)}$ are vector fields on $\mathbb P^3$ with codimension one singularities,
    and codimension one foliations tangent to them are linear pull-backs from foliations on $\mathbb P^2$.
\end{remark}

\begin{lemma}\label{L:dimTMd}
    Let $a,b,c$ be integers satisfying $0\le a \le b \le c \neq 0$, $\gcd(a,b,c)=1$, and $(a,b,c) \notin \{ (0,0,1), (1,1,1)\}$.
    If $d \ge 3$ and $V_d(a,b,c;n)$ contains a $1$-form with codimension two singularities
    then the dimension of $\TM_d(a,b,c;n)$ is equal to
    \[
        \dim \Aut(\mathbb P^3) + \dim \mathbb P(V_d(a,b,c;n)) - \dim \{ v \in H^0(\mathbb P^3, T_{\mathbb P^3}) ; [v,v_{(a,b,c)}] \wedge v_{(a,b,c)} =0 \}  \, .
    \]
\end{lemma}
\begin{proof}
    Let $G_{(a,b,c)} \subset \Aut(\mathbb P^3)$ be the subgroup that preserves the foliation defined by $v_{(a,b,c)}$.
    The natural action of $G_{(a,b,c)}$ on $\mathbb P H^0(\mathbb P^3, \Omega^1_{\mathbb P^3}(d+2))$ preserves
    $\mathbb P(V_d(a,b,c;n))$.  If $(a,b,c) \notin \{ (0,0,1),(1,1,1) \}$  then the Lie algebra of $G_{(a,b,c)}$ is equal to
    \[
        \{ v \in H^0(\mathbb P^3, T_{\mathbb P^3}) ; [v,v_{(a,b,c)}] \wedge v_{(a,b,c)} =0 \}.
    \]

    If $h \in \Aut(\mathbb P^3)$ and $[\omega] \in \mathbb P(V_d(a,b,c;n))$ then the fiber of $\Psi$
    over $\Psi(h,[\omega])$  contains the set
    \begin{equation}\label{E:set}
        \{ ( g^{-1} \cdot h, g^*[\omega] ) ; g \in G_{(a,b,c)} \} \, .
    \end{equation}

    Let $\omega \in V_d(a,b,c;n)$ be a $1$-form with codimension two singularities. Since $d\ge 3$, any vector field
    annihilating $\omega$ must be a complex multiple of  $v_{(a,b,c)}$. Thus, if $f^* [\omega] \in \mathbb P(V_d(a,b,c;n)) $ then $f$ must belong to $G_{(a,b,c)}$. It follows that the fiber over $\Psi(h,[\omega])$ must be contained in the Set (\ref{E:set}). The lemma follows.
\end{proof}

\begin{remark}\label{R:dim normalizer}
    If we exclude the cases where the vector field $v_{(a,b,c)}$ has codimension one zeros then the dimensions of the Lie algebras
    \[
        \mathfrak g_{(a,b,c)} =\{ v \in H^0(\mathbb P^3, T_{\mathbb P^3}) ; [v,v_{(a,b,c)}] \wedge v_{(a,b,c)} =0 \}
    \]
    have only three possible values.
    \begin{enumerate}
        \item If $1\le a < b < c$ then $\dim \mathfrak g_{(a,b,c)} = 3$.
        \item If $a=0$ and $b<c$, or $1\le a=b < c$, or $1\le a < b=c$ then $\dim \mathfrak g_{(a,b,c)} = 5$.
        \item If $a=0$ and $b=c=1$ then $\dim \mathfrak g_{(a,b,c)} = 7$.
    \end{enumerate}
\end{remark}

\subsection{The set of characteristic monomials}
We want to determine the $4$-uples $(a,b,c;n)$ such that $\TM_3(a,b,c;n)$ is an irreducible component
of $\Fol[3]{3}$. We will work on an affine coordinate system  $(x,y,z) \in \mathbb C^3 \subset \mathbb P^3$
where the action of $\mathbb C^*$ is tangent to the vector field $v = a x \frac{\partial}{\partial x} + b y \frac{\partial}{\partial y}
+ c z \frac{\partial}{\partial z}$.

In this coordinate system, a degree $d$ foliation on $\mathbb P^3$ is defined by a $1$-form $\omega$ that can be written as a
sum of homogeneous $1$-forms
\begin{equation}\label{E:Taylor}
    \omega = \omega_1 + \omega_2 + \omega_3 + \ldots + \omega_{d+1} + \omega_{d+2}
\end{equation}
where, for $i =1, \ldots, d+2$,  $\omega_i$ is a homogeneous $1$-form of degree $i$ (both $x,y,z$ and $dx,dy,dz$ have degree $1$) and the contraction of $\omega_{d+2}$ with the radial vector field $R =  x \frac{\partial}{\partial x} + y \frac{\partial}{\partial y}
+  z \frac{\partial}{\partial z}$ is equal to zero. If $\omega_{d+2} = 0$ then $i_R \omega_{d+1} \neq 0$ and the hyperplane at infinity
is invariant by the foliation defined by $\omega$.

To each $1$-form $\omega$ as in (\ref{E:Taylor}) we associate a subset $\chi(\omega) \subset \mathbb N^3$, called the set of characteristic monomials of $\omega$, by the following rule: $(i,j,k) \in \chi(\omega)$ if, and only if, for some $h \in \{ x,y,z\}$ the $1$-form $x^iy^jz^k \frac{dh}{h}$  appears with a non-zero coefficient in the Taylor expansion of $\omega$.

\begin{lemma}\label{L:trivial}
    Let $a,b,c \in \mathbb Z$ be integers satisfying $0 \le a \le b \le c \neq 0$ and $\gcd(a,b,c) = 1$.
    Let $\omega \in H^0(\mathbb P^3, \Omega^1_{\mathbb P^3}(d+2))$ be a non-zero $1$-form.
    If  $\omega$ belongs to $V_d(a,b,c;n)$ and we write it in the affine coordinates $(x_0:x_1:x_2:x_3) = (x:y:z:1)$
    then the set $\chi(\omega)$ is contained in the affine hyperplane
    \[
       \{ (\alpha, \beta, \gamma) \in \mathbb N^3 \, \vert \,  a \alpha+ b \beta + c \gamma = n \} \, .
    \]
    Reciprocally, if $\chi(\omega)$ is contained in the affine hyperplane above and $i_{v_{(a,b,c)}} \omega =0$,
    then $\omega \in V_d(a,b,c;n)$.
\end{lemma}
\begin{proof}
    If $v = a x \frac{\partial}{\partial x} + b y \frac{\partial}{\partial y} + c z \frac{\partial}{\partial z}$ and $h \in \{ x,y,z\}$
    then
    \[
        L_v \left( x^{\alpha} y^{\beta} z^\gamma \frac{dh}{h}\right) = ( a \alpha+ b \beta + c \gamma ) \cdot \left( x^{\alpha} y^{\beta} z^\gamma \frac{dh}{h} \right) \, .
    \]
    The lemma  follows.
\end{proof}

\subsection{First restrictions on the weights of the action}
For any $d \ge 3$, the proposition below, combined with Proposition \ref{P:belong to log}, reduces the search for irreducible components of $\Fol{d}$ of the form $\TM_d(a,b,c;n)$ to finitely many possibilities for $(a,b,c;n)$.

\begin{prop}\label{P:pif prop}
    Let $\omega$ be a $1$-form on $\mathbb C^3$ with codimension two zeros. If the foliation $\F$ on $\mathbb P^3$ defined by $\omega$ has degree $d \ge 3$ and $\chi(\omega)$ is contained in a line then $\F$ admits a polynomial integrating factor.
\end{prop}
\begin{proof}
    Let $H_1, H_2 \subset \mathbb R^3$ be two distinct hyperplanes such that $\chi(\omega) \subset H_1 \cap H_2$. If we write $H_i = \{ a_i x + b_i y + c_i z = \mu_i\}$ and $v_i = a_i x \frac{\partial}{\partial x}  + b_i y \frac{\partial}{\partial y}+ c_i z \frac{\partial}{\partial z}$ then
    \[
        L_{v_i} \omega = \mu_i \omega
    \]
    for $i = 1 ,2$. Therefore, both $v_1$ and $v_2$ belongs to $\aut(\F)$. Notice that $v_1\wedge v_2 \neq 0$  as otherwise $H_1$ and $H_2$ would be parallel hyperplanes with empty intersection. If both $v_1$ and $v_2$ belong to $\fix(\F)$ then $\F$ would be a foliation of degree at most two. Since this is not the case by assumption,  $v_1$ or $v_2$ does not belong to $\fix(\F)$. We apply Lemma \ref{L:fix neq aut} to conclude.
    \end{proof}

\begin{lemma}\label{L:longlist}
    Let $\omega$ be a general element of $V_3(a,b,c;n)$ where $0\le a\le b \le c \neq 0$,  $\gcd(a,b,c)=1$, and $(a,b,c) \notin \{ (0,0,1),(1,1,1) \}$. If $\omega$, seen as a section of $H^0(\mathbb P^3, \Omega^1_{\mathbb P^3}(5))$, has singular set of codimension $2$ and does not
    admit a polynomial integrating factor then $(a,b,c;n)$ is equal to one of the following $34$ possibilities:
    \begin{align*}
        &(0,1,1;2), (0,1,1;3),  (0,1,2;3), (1,1,2;5), (1,2,2;7),  (1,2,3;6), (1,2,3;7), \\
        &(1,2,3;8), (1,2,3;9),  (1,2,4;7), (1,2,4;9), (1,2,5;7),  (1,2,5;11),(1,2,5;12), \\
        &(1,3,4;7), (1,3,4;10), (1,3,4;13),(1,3,5;8), (1,3,5;11), (1,3,7;10),(1,4,6;13), \\
        &(2,3,4;11),(2,3,4;13), (2,3,5;11),(2,3,5;14),(2,3,7;16), (2,4,5;14),(2,4,5;17), \\
        &(2,5,6;17), (3,4,5;13),(3,4,5;14),(3,4,5;18),(4,5,7;19), (4,6,7;25) \, .
    \end{align*}
    Moreover, for any $(a,b,c;n)$ in the above list, the general $\omega \in V_3(a,b,c;n)$ defines a degree $3$ foliation
    on $\mathbb P^3$ without a polynomial integrating factor.
\end{lemma}
\begin{proof}
    Let $\Delta_{d} \subset \mathbb R^3$ be the closed subset $\{(\alpha, \beta, \gamma) \in \mathbb R^3 \, \vert \,  \alpha+ \beta + \gamma \le d+2, \alpha\ge 0, \beta \ge 0, \gamma \ge 0 \}$.

    The first step of the proof consists of an enumeration of all affine hyperplanes in $\mathbb R^3$ defined
    by equations $ax + by + c z = n$ with integral positive coefficients satisfying $0\le a\le b \le c \neq 0$ and
    $\gcd(a,b,c)=1$  that intersect $\Delta_{3}\cap \mathbb Z^3$ in at least three non-aligned points. For that, it suffices to consider
    three non-aligned points in $\Delta_{3}\cap \mathbb Z^3$, determine the equation for the hyperplane containing them, and check if
    they satisfy the condition $0\le a\le b \le c \neq 0$ and $\gcd(a,b,c)=1$.

    Having the hyperplanes list at hand, the second step of the proof consists of verifying if the general element of the corresponding $V_3(a,b,c,n)$ defines a $1$-form in $H^0(\mathbb P^3, \Omega^1_{\mathbb P^3}(5))$ with codimension two singularities and without polynomial integrating factor.

    We carried out the enumeration first by hand and double-checked our results by implementing the  procedure above in Maple.
\end{proof}

There is a certain amount of redundancy in the list above. For any $d \in \mathbb N$, consider the involution
\begin{align*}
        \iota_d : \mathbb Z^4 & \longrightarrow \mathbb Z^4 \\
        (\alpha,\beta,\gamma;\delta) & \mapsto (\gamma - \beta ,\gamma - \alpha,\gamma;\gamma(d+2) - \delta  ) \, .
\end{align*}

\begin{lemma}
     If $0 \le a \le b \le c \neq 0$ are positive integers satisfying $\gcd(a,b,c)=1$ then, for any $n \in \mathbb N$, the varieties $\TM_d(a,b,c;n)$ and $\TM_d( \iota_d(a,b,c;n))$ coincide.
\end{lemma}
\begin{proof}
    Let $\omega \in V_d(a,b,c;n)$,  $v=v_{(a,b,c)}$, and $w = cR - {v}$, where
    $R$ is the radial vector field on $\mathbb C^4$.
    Since $\omega$ is a homogeneous $1$-form of degree $d+2$ such that $L_{ v} \omega = n \omega$
    we see that
    \[
        L_{w } \omega = L_{cR -{v}} \omega = c L_R \omega - L_{{v}} \omega = c(d+2) \omega - n \omega \, .
   \]
   If $\varphi$ is the linear map $\varphi(x_0,x_1,x_2,x_3)=(x_1,x_0,x_3,x_2)$ then
   $\varphi_*  w=v_{(c-b,c-a,c)}$ and $\varphi^* \omega \in V_d(i_d(a,b,c;n))$.
    The equality between $\TM_d(a,b,c;n)$ and $\TM_d( \iota_d(a,b,c;n))$ follows.
\end{proof}

\subsection{Interpretation}\label{SS:interpretation}
Fix integers $a,b,c$  satisfying $1\le a \le b \le c$ and $\gcd(a,b,c) =1$.
Set $\mathbb P = \mathbb P(a,b,c)$ as the weighted projective plane of type $(a,b,c)$ in the
terminology of \cite{MR704986}. The weighted projective plane $\mathbb P$ is a singular
surface with at worst quotient singularities.

If we denote by $\Omega^{[1]}_{\mathbb P}$ the
sheaf of reflexive differentials on $\mathbb P$ ($\Omega^{[1]}_{\mathbb P} = \overline{\Omega^1_{\mathbb P}}$
in the notation of \cite[Section 2.1]{MR704986}) and by $\mathcal O_{\mathbb P}(n)$ the sheaf
of $\mathcal O_{\mathbb P}$-modules associated to the
module of quasi-homogeneous polynomials of degree $n$ (\cite[Section 1.4]{MR704986})
then the vector space $V_3(a,b,c;n)$ can be naturally identified with a subspace of
\[
    H^0(\mathbb P; \Omega^{[1]}_{\mathbb P}\otimes \mathcal O_{\mathbb P}(n)) = H^0(\mathbb P; \Omega^{[1]}_{\mathbb P}(n))   \,  .
\]
Therefore, a $1$-form   $\omega \in V_d(a,b,c;n)$ with zero set of codimension at least two defines a foliation  $\mathcal G$ on
$\mathbb P$ with normal sheaf $N_{\mathcal G} \simeq \mathcal O_{\mathbb P}(n)$ and tangent sheaf
$T_{\mathcal G} \simeq \mathcal O_{\mathbb P}((a+b+c)-n)$
and, at the same time, a codimension one foliation $\mathcal F$ on $\mathbb P^3$  with normal bundle $N_{\mathcal F} \simeq \mathcal O_{\mathbb P^3}(d+2-\epsilon)$, where $\epsilon$ is the vanishing order
along the hyperplane at infinity of the homogenization of $\omega$ as a homogeneous $1$-form on $\mathbb C^4$ of degree $d+2$ (i.e.
with homogenous coefficients of degree $d+1$). The foliation $\mathcal F$ on $\mathbb P^3$ coincides with the pull-back of the foliation $\mathcal G$ on $\mathbb P$ under the rational map
\begin{align*}
    \pi : \mathbb P^3 & \dashrightarrow \mathbb P \\
    (x_0:x_1:x_2:x_3) & \mapsto (x_0\cdot x_3^{a-1}: x_1 \cdot x_3^{b-1} : x_2\cdot x_3^{c-1}) \, .
\end{align*}

The result below is a specialization and adaptation to our particular context, of some general results scattered in the literature.

\begin{prop}\label{P:description}
    Notations as above. Let $\omega \in H^0(\mathbb P; \Omega^{[1]}_{\mathbb P}(n))$ be a non-zero $1$-form
    with codimension two singular set and let $\G$ be the foliation on $\mathbb P= \mathbb P(a,b,c)$ defined by $\omega$.
    The following assertions hold true.
    \begin{enumerate}
        \item\label{I:Miyaoka} If $n < a+b+c$ then $\mathcal G$ admits a rational first integral and its leaves are rational curves.
        \item\label{I:KF=0} If $n = a+b+c$ then $\mathcal G$ is defined by a closed rational $1$-form.
        \item\label{I:Darboux} If $n- (a+b+c)$ is strictly positive and does not belong to the semi-group generated by $a,b,c$ then one, and only one, of the following assertions is valid.
        \begin{enumerate}
            \item The foliation $\G$ has no algebraic leaves.
            \item The foliation $\G$ has exactly one algebraic leaf and is a virtually transversely additive foliation; or
            \item There exists an algebraic first integral $f:\mathbb P \dashrightarrow \mathbb P^1$ for $\G$ such that every fiber of $f$ has irreducible (but not necessarily reduced) support.
        \end{enumerate}
    \end{enumerate}
\end{prop}
\begin{proof}
    The sheaf $\Omega^1_{\mathcal G}$ of $1$-forms along the leaves of $\G$, i.e. the dual of $T_{\G}$,
    is isomorphic to $\mathcal O_{\mathbb P}(n-(a+b+c))$.

    The assumption of Item (\ref{I:Miyaoka}) is equivalent to $\Omega^1_{\mathcal G}\cdot H <0$ for any
    ample divisor. Consequently, Item (\ref{I:Miyaoka}) follows from Miyaoka's semi-positivity theorem.
    For a concrete proof in the case $a,b,c$ are two-by-two distinct that explicitly constructs
    the first integral see \cite[Proposition 2.2.1]{Lizarbe14}.

    The assumption of Item (\ref{I:KF=0}) is equivalent to the triviality of $T_{\mathcal G}$. Therefore,
    in this case $\G$ is defined by a global holomorphic vector field $w$ on $\mathbb P$. If the orbits of $w$ are algebraic
    then $\G$ admits a rational first integral and the differential of the first integral is a closed rational $1$-form
    defining $\G$. If instead the orbits of $w$ are not algebraic then $\fix(\G) \neq \aut(\G)$ (here we are considering
    natural analogues of the Lie algebras defined in Subsection  \ref{SS:inf vs pif}) and we can apply (the natural analogue
    of) Lemma \ref{L:fix neq aut} to produce a polynomial integrating factor for $\omega$ from a vector field
    $\xi \in \aut(\G) - \fix(\G)$. Item (\ref{I:KF=0}) follows.

    The assumption of Item (\ref{I:Darboux}) is equivalent to $h^0(\mathbb P,\mathcal O_{\mathbb P}(n-(a+b+c)))=0$. Assume $\G$ has one invariant curve $C$ and let $P \in \mathbb C[x,y,z]$ be the quasi-homogeneous polynomial
    of quasi-homogeneous degree $\delta$ defining it. Let also $w$ be the unique quasi-homogeneous vector field on $\mathbb C^3$
    such that
    \[
        d \omega = i_w (dx \wedge dy \wedge dz) \, .
    \]
    The vector field $w$ has quasi-homogeneous degree $\kappa = n-(a+b+c)$ and defines the foliation $\G$. In particular, $w$ has singularities of codimension at least two.
    The contraction of $w$ with $\frac{dP}{P}$ is a polynomial, because $C$ is $\G$-invariant, and  has quasi-homogeneous degree $\kappa$. Our assumption implies that $dP(w)=0$. Consequently, $d \omega$ and $\frac{dP}{P}\wedge \omega$ are proportional $2$-forms of the same quasi-homogeneous degree. Since $i_v d\omega = n \omega$
 and $i_v(\frac{dP}{P}\wedge \omega) = \delta \omega$, for $v = a x \frac{\partial}{\partial x} + b y \frac{\partial}{\partial y}
+ c z \frac{\partial}{\partial z}$,  we deduce that
    \[
        d \omega = \left( \frac{n}{\delta} \frac{dP}{P} \right) \wedge \omega
    \]
    Therefore $\G$ is a virtually transversely additive foliation whenever $\G$ admits an invariant curve.

    Assume now that $\G$ has two distinct invariant curves. If  irreducible quasi-homogeneous polynomials
    $P$ and $Q$ of quasi-homogeneous degrees $p$ and $q$ cut them out, then the logarithmic $1$-form
    \[
        \eta = q \frac{dP}{P} - p \frac{dQ}{Q}
    \]
    is such that $\eta(w) = 0$ (because $\eta(w)$  is a quasi-homogeneous polynomial of degree $\kappa$) and $\eta(v)=0$.
    It follows that $P^q/Q^p$ is a rational first integral for $\G$.

    Let $f: \mathbb P \dashrightarrow \mathbb P^1$ be the Stein factorization of any rational first integral for $\G$. If one of the fibers of $f$ does not have irreducible support,
       we can apply the argument above to produce a first integral non-constant along a general fiber of $f$. Indeed, if $P$ and $Q$ are now irreducible quasi-homogeneous polynomials cutting out two distinct irreducible components of the same fiber of $f$ then $h = P^q/Q^p$ is a rational first integral for $\G$ that assumes two distinct values, $0$ and $\infty$, on the fiber
    of $f$ containing $\{P=0\}$ and $\{Q=0\}$. By continuity, $h$ is not constant on the general fiber of $f$.
    This is impossible because the general fiber of $f$ is irreducible and any rational first integral must be constant along it. Item (\ref{I:Darboux}) follows.
\end{proof}

\begin{remark}
    The  proof of Item (\ref{I:Darboux}) of Proposition \ref{P:description} can be phrased in more intrinsic terms, along the lines of
    the proof of \cite[Theorem 1]{MR4150930}. Let $U \subset \mathbb P$ be the smooth locus. Since $C$ is $\G$-invariant, the logarithmic differential  $-\frac{dP}{P}$ defines a flat logarithmic connection with trivial monodromy on  $\mathcal O_{\mathbb P}(\delta)_{|U}$ that restricts to a holomorphic $\G$-partial connection on the line-bundle $\mathcal O_{\mathbb P}(\delta)_{|U}$. Likewise, $-\frac{n}{\delta} \frac{dP}{P}$ defines a flat logarithmic connection $\nabla^{C}$ with finite monodromy on ${N_{\G}}_{|U}$ that restricts to a holomorphic $\G$-partial connection. If $\nabla$ is Bott's partial connection for $\G$ then $\nabla - \nabla^{C}_{|T_{\G}}$ defines a holomorphic section $\sigma$ of ${\Omega^1_{\G}}_{|U}$. Since $\mathbb P$ is normal and $\mathbb P - U$ has codimension two, $\sigma$ extends to a section of $\Omega^1_{\G}$. By assumption, $\sigma=0$. This shows that Bott's partial connection extends to a flat logarithmic connection with finite monodromy. Consequently $\G$ is virtually transversely additive.
\end{remark}

\subsection{Foliations with split tangent sheaf and  finitely many non-Kupka singularities}
The first irreducible components of the form $\TM_d(a,b,c;n) \subset \Fol[3]{d}$ were presented in \cite{CaCeGiLi}, where
they assumed $1\le a < b < c$ and that the general element of $\TM_d(a,b,c;n)$ has split tangent sheaf and finitely many non-Kupka singularities. Recall that a singularity $p$ of a foliation is a Kupka singularity if $d \omega(p) \neq 0$ for
any local $1$-form $\omega$ with codimension two singularities defining the foliation in a neighborhood of $p$. In the literature, irreducible components of the space of foliations having general elements possessing  the above properties are also called generalized Kupka components.

\begin{prop}\label{P:codim 3}
    Let $d\geq 3$ and $0 \le a \le b \le c \neq 0$ be integers satisfying $\gcd(a,b,c) = 1$ and $(a,b,c)\neq (1,1,1)$. If the general $ \omega \in V_d(a,b,c;n)$ defines a degree $d$ foliation with tangent sheaf isomorphic to $\mathcal O_{\mathbb P^3} \oplus \mathcal O_{\mathbb P^3}(2-d)$
   and with finitely many non-Kupka singularities then $\TM_d(a,b,c;n)$ is an irreducible component of $\Fol{d}$.  	
\end{prop}
\begin{proof} Let $\F$ be the foliation of degree $d$ defined by $\omega$.
   According to \cite[Corollary 1]{CukiermanPereira} any foliation $\G$
   sufficiently close to $\F$ has tangent sheaf isomorphic to $\TF$.

   We claim that $\mathfrak{fix}(\G) = \mathfrak{aut}(\G)$. Aiming at a contradiction, assume this is not the case.
   Lemma \ref{L:fix neq aut} implies $\G$ admits a polynomial integrating factor. Proposition \ref{P:belong to log}
   implies that $\G$ belongs to one of the logarithmic components. On the other hand, \cite[Theorem 3]{CuSoVa}
   guarantees that the generic element of the logarithmic foliation of degree greater than or equal to three on $\Proj^3$ has at least one isolated
   singularity and thus its tangent sheaf is not split, giving the sought contradiction.

  The trivial factor $\Ol_{\Proj^3}$ corresponds to a non-zero vector field $v \in \mathfrak{fix}(\G)$ with singularities
  of codimension at least two.
   Moreover, $\deg \G \geq 3$ implies that $\mathfrak{fix}(\G)$ is equal to $\mathbb C\cdot v$ and
  $\mathfrak{fix}(\G) = \mathfrak{aut}(\G)$. Thus the orbits of $v$ are algebraic. Since $\mathbb C v$ is a small deformation of $\mathbb Cv_{(a,b,c)}$ with algebraic orbits, the two Lie algebras must be conjugated. The proposition follows.
\end{proof}

\begin{cor}
    Let $d\geq 3$ and $1\le a < b < c$ be integers satisfying $\gcd(a,b,c)=1$. If the general $ \omega \in V_d(a,b,c;n)$ defines a degree $d$ foliation with finitely many non-Kupka singularities then $\TM_d(a,b,c;n)$ is an irreducible component of $\Fol{d}$.
\end{cor}
\begin{proof}
    The assumption on $a,b,c$ implies that $v_{(a,b,c)}$ has isolated singularities. Therefore \cite[Lemma 4.2]{LorayPereiraTouzet13}
   implies that the tangent sheaf of $\F$ splits as the direct sum $\Ol_{\Proj^3}\oplus    \Ol_{\Proj^3}(2-d)$ and the
   result follows from Proposition \ref{P:codim 3}.
\end{proof}

The theorem below presents a complete characterization of the irreducible components of the form $\TM_d(a,b,c;n) \subset \Fol[3]{d}$ with $1\le a < b < c$ and finitely many non-Kupka singularities. It is a rephrasing of \cite[Theorem B.1]{Constant} using the notation of the present work.

\begin{thm}\label{T:Raphael}
    Let $d \geq 3$ and $1\le a<b<c$ be integers satisfying $\gcd(a,b,c)=1$. The general element of $V_d(a,b,c;n)$ defines a foliation with finitely many non-Kupka singularities if, and only if, one of the following conditions hold for
    $(\alpha,\beta,\gamma;\delta)=(a,b,c;n)$ or $(\alpha,\beta,\gamma;\delta)=\iota_d(a,b,c;n)$:
    \begin{enumerate}
        \item\label{I:R1} There exists an integer $r<d-1$ such that
        \[
            (\alpha,\beta,\gamma;\delta) = (r,r+1,d;d(r + 1) + 2r+1);
        \]
        \item\label{I:R2} There exists an integer $k$ that divides $d+1$ and an integer $m$, relatively prime to $k$, such that
        \[
            (\alpha,\beta,\gamma;\delta) = ( md,md+k,kd; md^2 + 2md + kd + k);
        \]
        \item\label{I:R3} The integer $\gamma$ divides $d^2$ or $d^2+d+1$, and there exists a positive integer $m$, relatively prime to $\gamma$,
        such that
        \[
            (\alpha,\beta,\gamma;\delta) = (md, m(d+1),\gamma; m(d^2+2d+1)+\gamma);
        \]
        \item\label{I:R4} The integer $\gamma$ divides $d^2$, or $d^2-1$, or $d^2-d$, and there exists a positive integer $m$, relatively prime to $\gamma$,
        such that
        \[
            (\alpha,\beta,\gamma;\delta) = (m(d-1), md, \gamma ; m(d^2+d-1)+\gamma).
        \]
    \end{enumerate}
    In each of these cases, $TM_d(a,b,c;n)$ is an irreducible component of $\Fol[3]{d}$.
\end{thm}

Similarly, the following corollary of Theorem \ref{T:Raphael} is a rephrasing of \cite[Corollary 4.8]{Constant} in the notation
of the present paper.

\begin{cor}\label{C:Raphael}
    Assume that $a,b,c$ are integers satisfying $\gcd(a,b,c)=1$ and $1 \le a < b< c$. If $\omega \in V_3(a,b,c;n)$ is a $1$-form defining a degree $3$ foliation with finitely many non-Kupka singularities then $(a,b,c;n)$ belongs to the following list of $12$ possibilities:
    \begin{align*}
        & (1,2,3;7),(1,2,3;8),(1,2,4;7),(1,2,4;9), (1,3,7;10),(1,4,6;13), \\
        & (2,3,4;11),(2,3,4;13),(2,3,7;16),(2,5,6;17),(4,5,7;19),(4,6,7;25) \, .
    \end{align*}
    In particular, for any $(a,b,c;n)$ in the list above, $\TM_3(a,b,c;n)$ is an irreducible
    component of $\Fol{3}$.
\end{cor}

\begin{remark}\label{R:Ruben1}
    A complete characterization of irreducible components of the form $\TM_d(a,b,c;n)$ with split tangent sheaf, finitely many non-Kupka singularities and with $0<a = b < c$ or $0<  a < b = c$  is not available in the literature. Nevertheless \cite[Theorem 1]{Lizarbe17} proves that for every  $q\ge 1$ and every pair $a,c$  of relatively prime positive integers satisfying $a<c$ and $c\ge 3$ the sets
    \[
        \TM_{q c +1}(a,a,c;a(qc-1) + 2a + c)    \quad \text{and} \quad \TM_{qc +2}(a,a,c;acq + 2a +c)
    \]
    are irreducible components of the corresponding spaces of foliations such that the general member is a foliation with split tangent sheaf and finitely many non-Kupka singularities. None of these examples have degree $3$.
\end{remark}

\begin{lemma}\label{L:extraKupka}
    Let $a,b,c$ be integers satisfying $0\le a \le b \le c \neq 0$, $\gcd(a,b,c)=1$, and $(a,b,c) \notin \{ (0,0,1), (1,1,1)\}$.
    If $\omega \in  V_3(a,b,c;n)$ is a $1$-form defining a degree $3$ foliation with finitely many non-Kupka singularities and without polynomial integrating factor then $(a,b,c;n)$ belongs to the list presented in Corollary \ref{C:Raphael} or to the following list of $3$ extra possibilities:
    \begin{align*}
        (0,1,2;3), (1,1,2;5), (1,2,2;7).
    \end{align*}
    The general element of $\TM_3(0,1,2;3) = \TM_3(1,2,2;7)$ is a foliation with  tangent sheaf isomorphic to $\mathcal O_{\mathbb P^3} \oplus \mathcal O_{\mathbb P^3}(-1)$,
        while  the general element of $\TM_3(1,1,2;5)$ defines a foliation whose tangent sheaf is not locally free.
\end{lemma}
\begin{proof}
    A case-by-case analysis shows that the general element of $V_3(a,b,c;n)$ where $(a,b,c;n)$ belongs to the list presented in Lemma \ref{L:longlist} and $a=0$ or $a=b$ or $b=c$, has finitely many Kupka singularities if, and only if, $(a,b,c;n)$ belongs to the set
    $\{ (0,1,2;3), (1,1,2;5), (1,2,2;7)\}$.

    In the affine chart $(x : y: z :1)$,
        for any $\omega \in V_3(a,b,c;n)$ there exists $w$ such that $\omega$ is, up to a constant, equal to $i_v i_w (dx \wedge dy \wedge dz)$. Indeed, since $i_v \omega =0$ the expression $L_v \omega = n \omega$ is equivalent to $i_v d\omega = n \omega$. If we write $(1/n) d \omega  = i_w (dx \wedge dy \wedge dz)$ then $\omega = i_v i_w (dx \wedge dy \wedge dz)$.
    Therefore, if the singular set of $\omega$ has codimension at least two,  the tangent sheaf of the foliation defined by $\omega$ is free on the open subset of $\mathbb P^3$ given by $x_3 \neq0$. Likewise, the same holds for $x_2 \neq 0$ since the weights of the action in this open subset all have the same sign.

    On the open subset $x_0 \neq 0$, $(a,b,c;n)$ transforms
    to $(b-a,c-a,-a; n - a(d+2) )$
    , and on the open subset $x_1 \neq0$, $(a,b,c;n)$ transforms to $(a-b,c-b,-b; n - b(d+2))$. Hence, if $n \neq a(d+2)$ and $n \neq b(d+2)$ and $\omega$ has singularities of codimension at least two, then  $\ker \omega$ is a locally free subsheaf of $T_{\mathbb P^3}$. Moreover,  the inclusion
    of $\mathcal O_{\mathbb P^3}$ into $\ker \omega$ induced by $v$ has locally free cokernel. Therefore $\ker \omega$ fits into the
    exact sequence
    \[
        0 \to \mathcal O_{\mathbb P^3} \to \ker \omega \to \mathcal O_{\mathbb P^3}(-1) \to 0 \, .
    \]
    Since $H^1(\mathbb P^3,\mathcal O_{\mathbb P^3}(1))=0$ (no line-bundle on $\mathbb P^n$ has intermediate cohomology),  this sequence splits showing that $\ker \omega$ is isomorphic to
    $\mathcal O_{\mathbb P^3} \oplus \mathcal O_{\mathbb P^3}(-1)$. We can apply this argument for $ (a,b,c;n) \in \{ (0,1,2;3),  (1,2,2;7) \}$.

    When $(a,b,c;n) = (1,1,2;5)$, the above argument does not apply since $n = a(d+2)=b(d+2)$. Moreover, the general element of $V_3(1,1,2;5)$  has an isolated singularity at the line $\{x_2 = x_3=0\}$ preventing it to define a foliation with locally free tangent sheaf.
\end{proof}

\begin{prop}\label{C:extraKupka}
    The set $\TM_3(0,1,2;3) = \TM_3(1,2,2;7)$ is an irreducible component of $\Fol{3}$.
\end{prop}
\begin{proof}
    Combine Lemma \ref{L:extraKupka} with Proposition \ref{P:codim 3}.
\end{proof}

\tikzset{
    table/.style={
        matrix of nodes,
        row sep=-\pgflinewidth,
        column sep=-\pgflinewidth,
        nodes={
            rectangle,
            draw=black,
            align=center
        },
        minimum height=1.5em,
        text depth=0.5ex,
        text height=2ex,
        nodes in empty cells,
        every even row/.style={
            nodes={fill=gray!20}
        },
        column 1/.style={
            nodes={text width=4.7cm,align=left,font=\small}
        },
        column 2/.style={
            nodes={text width=0.8cm,align=center},font=\small},
        column 3/.style={
            nodes={text width=0.8cm,align=center},font=\small},
        column 4/.style={
            nodes={text width=4.5cm,align=left, font=\small}
        },
        row 1/.style={
            nodes={ font=\bfseries
            }
        }
    }
}
\begin{table}[H]
\begin{tikzpicture}
\matrix (mat) [table,text width=3cm]
{
    {\bf Irreducible Component} & ${\dim}$     &  $\mathrm{Zdim}$      & Comment   \\
    $\TM_3(0,1,2;3) = \TM_3(1,2,2;7) $  & $16$     & $16$  & Previously unkown example.    \\
    $\TM_3(1,2,3;7)=\TM_3(1,2,3;8)$  & $16$     & $16$  &   \\
    $\TM_3(1,2,4;7)=\TM_3(2,3,4;13)$  & $15$     & $15$  &  Closed rational $1$-form. \\
    $\TM_3(1,2,4;9)=\TM_3(2,3,4;11)$  & $15$     & $15$  &  \\
    $\TM_3(1,3,7;10)=\TM_3(4,6,7;25)$ & $14$     & $14$  &  Rational fibration. \\
    $\TM_3(1,4,6;13)= \TM_3(2,5,6;17)$ & $14$     & $14$  &  \\
    $\TM_3(2,3,7;16)= \TM_3(4,5,7;19)$ & $14$     & $14$  &  \\
};
\end{tikzpicture}
\caption{\small The seven irreducible components of $\Fol{3}$ tangent to a multiplicative action, with split tangent sheaf, and finitely many non-Kupka singularities}\label{Tab:split}
\end{table}

We summarize the classification of irreducible components of $\Fol{3}$ tangent to a multiplicative action, split tangent sheaf, and finitely many non-Kupka singularities in Table \ref{Tab:split}.  Notice that for all the irreducible components listed in Table \ref{Tab:split}, the dimensions coincide with the dimensions of the Zariski tangent space of $\Fol{3}$ at a generic element of the respective component. Thus all these irreducible components are generically reduced. Six of the irreducible components listed in Table \ref{Tab:split} were previously known. Only $\TM_3(0,1,2;3)= \TM_3(1,2,2;7)$ is a previously unknown irreducible component. The general element of $\TM_3(1,2,2;7)$ is a transversely
affine foliation (rational pull-back of a Ricatti foliation having one algebraic section invariant). Every element of $\TM_3(1,3,7;10)$ is a rational fibration according to Item (\ref{I:Miyaoka}) of Proposition \ref{P:description} and every element of $\TM_3(1,2,4;7)$ is defined
by a closed rational $1$-form according to Item (\ref{I:KF=0}) of the same proposition.

\subsection{Foliations without rational first integral}
We now turn our attention to the irreducible components of $\Fol{3}$ tangent to a multiplicative action and
without a rational first integral.

\begin{prop}\label{P:toolmult}
    If the set $\TM_3(a,b,c;n)$ contains a foliation which does not admit a rational first
    integral and does not admit a polynomial integrating factor then $\TM_3(a,b,c;n)$ is an irreducible
    component of $\Fol{3}$.
\end{prop}
\begin{proof}
    Let $\F$ be a foliation in $\TM_3(a,b,c;n)$ satisfying the assumptions of the proposition.
    Let $\F_{\varepsilon}$, $\varepsilon \in \Delta$, be a deformation of $\F$ parameterized by the unit disc $\Delta\subset \mathbb C$. Lemma \ref{L:pifclosed} guarantees $\F_{\varepsilon}$  still does not admit a polynomial integrating factor for $\varepsilon$ sufficiently small. Moreover, if $\varepsilon$ is very general then
    $\F_{\varepsilon}$ does not admit a rational first integral. Therefore we can apply  \cite[Theorem A]{LorayPereiraTouzet17} to
    deduce that $\F_{\varepsilon}$ is tangent to an algebraic action when $\varepsilon$ is sufficiently small and very general.
    Thus, the deformed action weights do not vary, and $\F_{\varepsilon}$ is tangent to a $\mathbb C^*$--action which is conjugated by an element of $\Aut(\mathbb P^3)$ to the original one.
\end{proof}

\subsubsection{Non-rigid foliations without rational first integral}
Recall from the proof of Lemma \ref{L:longlist}, that  $\Delta_{d} \subset \mathbb R^3$
is equal to the closed subset $\{(\alpha, \beta, \gamma) \in \mathbb R^3 \, \vert \,  \alpha+ \beta + \gamma \le d+2, \alpha\ge 0, \beta \ge 0, \gamma \ge 0 \}$. Denote the (topological) boundary of $\Delta_d$ by $\partial \Delta_d$.

\begin{lemma}\label{L:nonrigid implies nonintegrable}
    Let $\omega \in  V_d(a,b,c;n)$. If $\chi(\omega) \cap \left( \Delta_d - \partial \Delta_d \right) \neq \emptyset$
    then the general element of $V_d(a,b,c;n)$ does not admit a rational first integral.
\end{lemma}
\begin{proof}
    The assumption implies that $V_d(a,b,c;n)$ contains a $1$-form $\omega_0$  with $\chi(\omega_0) = \{(i,j,k)\}$
       satisfying
    $i +j + k < d+2$, $i\neq 0$, $j\neq 0$, and $k \neq 0$. Notice that the general $\omega_0$ having this property can be written on affine coordinates as
    \[
        \omega_0 = x^i y^j z^k  \left( \alpha \frac{dx}{x} + \beta \frac{dy}{y} + \gamma \frac{dz}{z} \right)
    \]
    where $\alpha, \beta, \gamma$ are complex numbers satisfying the relation $a \alpha + b \beta + c \gamma=0$.

    If $\alpha, \beta, \gamma$ are sufficiently general then $\omega_0$ defines a degree $2$ foliation on $\mathbb P^3$
    without rational first integrals (it suffices to assume that the quotients of any two of these three numbers are not rational).
    It follows that any sufficiently small deformation $\omega_0$ also does not have a rational first integral.
\end{proof}

\begin{cor}\label{C:nonintegrablenonrigid}
    If $(a,b,c;n)$ is equal to one of the following $14$ possibilities
    \begin{align*}
        &(0,1,1;2) , (0,1,1;3),
        (0,1,2;3),
        (1,1,2;5),
        (1,2,2;7),
        (1,2,3;6), (1,2,3;7),  \\ & (1,2,3;8),(1,2,3;9),
        (1,2,4;7), (1,2,4;9), (1,3,4;10), (2,3,4;11), (2,3,4;13),
    \end{align*}
    then the general element of $V_3(a,b,c;n)$ does not admit a rational first integral and does not admit
    a polynomial integrating factor.
\end{cor}
\begin{proof}
    We can apply Lemma \ref{L:nonrigid implies nonintegrable} to prove the corollary for
    all cases except $(a,b,c;n)= (1,3,4;10)$. An explicit computation
    shows that a general element in $V_3(1,3,4;10)$ defines a foliation on $\mathbb P(1,3,4)$
    having a singularity  with invertible linear part and non-rational quotient of eigenvalues.
\end{proof}

\begin{prop}
    The eight sets presented in Table \ref{Tab:nonintegrablenonrigid} are irreducible components of $\Fol{3}$.
\end{prop}
\begin{proof}
    Combine Lemma \ref{L:longlist}, Proposition \ref{P:toolmult},  and Corollary \ref{C:nonintegrablenonrigid}.
\end{proof}

We observe that  \cite[Theorem 2 of Chapter 2]{Lizarbe14} implies that the general element of $\TM_3(1,1,2;5)$ has no
algebraic leaves.

Contrary to what happens in  Tables \ref{Tab:LOG} and \ref{Tab:split}, not every irreducible component appearing in Table \ref{Tab:nonintegrablenonrigid} is generically reduced. Besides the four irreducible components that also appear in Table \ref{Tab:split}, there is just one extra component that is generically reduced: $\TM_3(1,1,2;5)$. Recall from Lemma \ref{L:extraKupka} that the general element of this irreducible component is a foliation with finitely many non-Kupka singularities. Each of the other three irreducible components are not generically reduced. Notice that any foliation parameterized by them has a curve of non-Kupka singularities.

\tikzset{
    table/.style={
        matrix of nodes,
        row sep=-\pgflinewidth,
        column sep=-\pgflinewidth,
        nodes={
            rectangle,
            draw=black,
            align=center
        },
        minimum height=1.5em,
        text depth=0.5ex,
        text height=2ex,
        nodes in empty cells,
        every even row/.style={
            nodes={fill=gray!20}
        },
        column 1/.style={
            nodes={text width=4.50cm,align=left,font=\small}
        },
        column 2/.style={
            nodes={text width=0.7cm,align=center},font=\small},
        column 3/.style={
            nodes={text width=0.8cm,align=center},font=\small},
        column 4/.style={
            nodes={text width=5.47cm,align=left, font=\small}
        },
        row 1/.style={
            nodes={ font=\bfseries
            }
        }
    }
}

\begin{table}
\begin{tikzpicture}
\matrix (mat) [table,text width=3.85cm]
{
    {\bf Irreducible Component} & ${\dim}$     &  $\mathrm{Zdim}$      & Comment   \\
    $\TM_3(0,1,1;2) = \TM_3(0,1,1;3) $ & ${17}$    & ${20}$  & Riccati foliation.   \\
    $\TM_3(0,1,2;3) = \TM_3(1,2,2;7) $  & $16$     & $16$  & Appears in Table \ref{Tab:split}.   \\
    $\TM_3(1,1,2;5)$  & $21$     & $21$  & General element without algebraic leaves. \\
    $\TM_3(1,2,3;6)=\TM_3(1,2,3;9)$  & ${15}$    & ${19}$  &    \\
    $\TM_3(1,2,3;7)=\TM_3(1,2,3;8)$  & $16$     & $16$  &   Appears in Table \ref{Tab:split}\\
    $\TM_3(1,2,4;7)=\TM_3(2,3,4;13)$  & $15$     & $15$  &  Appears in  Table \ref{Tab:split}. \\
    $\TM_3(1,2,4;9)=\TM_3(2,3,4;11)$  & $15$     & $15$  &  Appears in Table \ref{Tab:split} \\
    $\TM_3(1,3,4;10)$ & ${15}$     & ${17}$  &  \\
};
\end{tikzpicture}
\caption{\small The eight non-rigid irreducible components of $\Fol{3}$ with general elements tangent to a multiplicative action and without a rational first integral. Note that four of them already appeared in Table \ref{Tab:split}.}\label{Tab:nonintegrablenonrigid}
\end{table}

\subsubsection{Rigid foliations without rational first integral}

\begin{lemma}\label{L:nonintegrablerigid1}
    If $(a,b,c;n)$ is equal to one of the following $8$ possibilities
    \begin{align*}
        &(1,2,5;11), (1,2,5;12),  (1,4,6;13) ,  (2,3,7;16),\\
        &(2,5,6;17), (3,4,5;13), (3,4,5;14), (4,5,7;19) ,
    \end{align*}
    then the foliation on  $\mathbb P(a,b,c)$ defined by a  general element of $V_3(a,b,c;n)$ has
    a singularity with invertible linear part and non-rational quotient of eigenvalues. In particular,
    the general element in $V_3(a,b,c;n)$  does not admit a rational first integral.
\end{lemma}
\begin{proof}
    Direct computation with a general element in each of the sets $V_3(a,b,c;n)$.
\end{proof}

\begin{lemma}\label{L:nonintegrablerigid2}
   If $(a,b,c;n) \in \{ (2,3,5;11) ,(2,3,5;14) \}$ then the general element of the set $V_3(a,b,c;n)$ defines
   a virtually transversely additive foliation which does not admit a rational first integral.
\end{lemma}
\begin{proof}
    We will consider the model $\TM_3(2,3,5;11)$. The set $\TM_3(2,3,5;11)$ consists in the closure
    of the orbit the foliation $\F$ on $\mathbb P^3$ defined on affine coordinates by
    \[
        \omega  =  (5 x^2z -3y^3)   dx +  (2xy^2-5yz) dy   +  ( 3y^2 -2x^3) dz  .
    \]
    We will denote the corresponding foliation on $\mathbb P = \mathbb P(2,3,5)$  by $\G$.

    From the explicit equation of $\omega$, it is clear that the
    curve $\{ 2x^3 - 3y^2  =0\}$ is  $\G$-invariant. Proposition \ref{P:description} implies that $\G$ is a virtually transversely additive foliation and, therefore, so is the foliation on $\mathbb P^3$ determined by it.

    There are no local obstructions for the existence of a rational first integral for $\omega$, i.e., the quotients of eigenvalues
    of the singularities of $\G$ lie in $\mathbb Q$. In order to verify that $\G$ does not admit a rational first integral, we argue
    by contradiction.

    Assume $\G$ admits a rational first integral and let $\varphi : \mathbb P \dashrightarrow \mathbb P^1$ be the Stein factorization of it.
    Item (\ref{I:Darboux}) of Proposition \ref{P:description} guarantees that every fiber of $\varphi$ has  irreducible support, but is not necessarily reduced. Hence,  a fiber of $\varphi$ is non-reduced if, and only if, it is a  multiple fiber. We claim that $\varphi$ has at most two non-reduced fibers. Indeed, we can argue as in the proof of \cite[Theorem 3.3]{LorayPereiraTouzet13} (a version of Lins Neto's proof of a classical result by Halphen) replacing
    the radial vector field $R = x\frac{\partial}{\partial x} +  y\frac{\partial}{\partial y} + z\frac{\partial}{\partial z}$ by    $v = a \frac{\partial}{\partial x} + b \frac{\partial }{\partial y} + c \frac{\partial }{\partial z}$.

    If we place the possible multiple fibers of $\varphi$ over $[0:1]$ and $[1:0]$ then there exist
    irreducible quasi-homogeneous polynomials $f,g \in \mathbb C[x,y,z]$ such that $\varphi$ is, in quasi-homogeneous coordinates, equal to
    \begin{equation*}
         \left( \frac{f^{\deg(g)}}{g^{ \deg(f)}}\right)^{\frac{1}{\gcd(\deg(f),\deg(g))}} ,
    \end{equation*}
    where $\deg(\cdot)$ denotes the quasi-homogeneous degree with respect to the weights $(2,3,5)$.
    Since $\varphi$ has at most two non-reduced fibers and they are supported on $\{ f=0\} \cup \{ g=0\}$, it follows that the $1$-form
    \[
         \omega' = \deg(f) fdg - \deg(g) gdf
    \]
    has singular set of codimension at least two. Because $\omega'$ defines the foliation $\G$, it must coincide with $\omega$ up to a multiplicative constant.

    As  $\omega$ has quasi-homogeneous degree $11$, we have that $\deg(f) + \deg(g) = 11$. Since any reduced and irreducible fiber
    of $\varphi$ has degree at least $\max(\deg(f),\deg(g))$, $\deg(f)$ or $\deg(g)$  must be equal to $6= \deg(2x^3 - 3y^2 )$.
    So, if $\deg(f)=6$ then $\deg(g)=5$. In particular, $g$ lies in the two dimension vector space generated by $xy$ and $z$.
    An easy verification shows that there does not exist any polynomial in this vector space defining an invariant curve for $\G$.
    Therefore $\G$ does not admit a rational first integral, and the same holds for the general element of $\TM_3(2,3,5;11)$.
\end{proof}

In all cases listed in the statements of Lemma \ref{L:nonintegrablerigid1} and of Lemma \ref{L:nonintegrablerigid2}, $V_3(a,b,c;n)$ is a three dimensional vector space and $1\le a<b<c$. It follows from Lemma \ref{L:dimTMd} and Remark \ref{R:dim normalizer} that the corresponding irreducible components all have dimension $14$.

\tikzset{
    table/.style={
        matrix of nodes,
        row sep=-\pgflinewidth,
        column sep=-\pgflinewidth,
        nodes={
            rectangle,
            draw=black,
            align=center
        },
        minimum height=1.5em,
        text depth=0.5ex,
        text height=2ex,
        nodes in empty cells,
        every even row/.style={
            nodes={fill=gray!20}
        },
        column 1/.style={
            nodes={text width=4.65cm,align=left,font=\small}
        },
        column 2/.style={
            nodes={text width=0.8cm,align=center},font=\small},
        column 3/.style={
            nodes={text width=0.8cm,align=center},font=\small},
        column 4/.style={
            nodes={text width=5cm,align=left, font=\small}
        },
        row 1/.style={
            nodes={ font=\bfseries
            }
        }
    }
}

\begin{table}
\begin{tikzpicture}
\matrix (mat) [table,text width=3cm]
{
    {\bf Irreducible Component} & ${\dim}$     &  $\mathrm{Zdim}$      & Comment   \\
    $\TM_3(1,2,5;11)= \TM_3(3,4,5;14) $ & ${14}$     & ${19}$  &  \\
    $\TM_3(1,2,5;12)= \TM_3(3,4,5;13)$ & ${14}$     & ${19}$  &  \\
    $\TM_3(1,4,6;13)= \TM_3(2,5,6;17)$ & $14$     & $14$  & Appears in Table \ref{Tab:split} \\
    $\TM_3(2,3,5;11)=\TM_3(2,3,5;14)$ & ${14}$     & ${19}$  & Virtually transversely additive. \\
    $\TM_3(2,3,7;16)=\TM_3(4,5,7;19)$ & $14$     & $14$  &  Appears in Table \ref{Tab:split}\\
};
\end{tikzpicture}
\caption{\small Five rigid, but not necessarily infinitesimally rigid, irreducible components of $\Fol{3}$ with general element tangent to a multiplicative action and without a rational first integral. Two of them already appeared in Table \ref{Tab:split}.}\label{Tab:nonintegrablerigid}
\end{table}

\subsection{Candidates with rational first integral}

\begin{prop}\label{P:rationalfirstintegral}
    The general element of $\TM_3(a,b,c;n)$ admits a rational first integral,  does not admit a polynomial integrating
    factor, and has singularities of codimension at least two  if, and only if, the quadruple $(a,b,c;n)$ belongs to the set
    \begin{align*}
            \{ & (1,2,5;7),(1,3,4;7),(1,3,4;13),(1,3,5;8),(1,3,5;11), \\ & (1,3,7;10),(2,4,5;14),(2,4,5;17),(3,4,5;18),(4,6,7;25) \} .
        \end{align*}
    Moreover, the following assertions hold true.
    \begin{enumerate}
        \item The set  $\TM_3(1,3,7;10)=\TM_3(4,6,7;25)$  is an irreducible and generically reduced component of $\Fol{3}$.
        \item Both $\TM_3(1,2,5;7) = \TM_3(3,4,5;18)$  and $\TM_3(1,3,5;11) =  \TM_3(2,4,5;14)$  are
        subsets of the irreducible component $\SLog(2,5)$ of $\Fol{3}$.
        \item The set  $\TM_3(1,3,4;7)=\TM_3(1,3,4;13)$  is contained in the irreducible component $\SLog(3,4)$ of $\Fol{3}$.
        \item The set $\TM_3(1,3,5;8) = \TM_3(2,4,5;17)$ is not contained in any of the  irreducible components presented in
        Tables \ref{Tab:LOG}, \ref{Tab:split}, \ref{Tab:nonintegrablenonrigid}, and \ref{Tab:nonintegrablerigid}.
    \end{enumerate}
\end{prop}
\begin{proof}
    Let $\omega \in V_3(a,b,c;n)$ be a sufficiently general element.  From the $34$ possibilities for $\{ (a,b,c;n)\}$ given by  Lemma \ref{L:longlist}, we already analyzed
    \[
          \underbrace{12}_{\ref{C:Raphael}}
        + \underbrace{3}_{\ref{L:extraKupka}}
        + \underbrace{14}_{\ref{C:nonintegrablenonrigid}}
        + \underbrace{8}_{\ref{L:nonintegrablerigid1}}
        + \underbrace{2}_{\ref{L:nonintegrablerigid2}}
        - \underbrace{6}_{ \ref{C:Raphael} \cap \ref{C:nonintegrablenonrigid}  }
        - \underbrace{3}_{\ref{L:extraKupka} \cap \ref{C:nonintegrablenonrigid}}
        - \underbrace{4}_{ \ref{C:Raphael} \cap \ref{L:nonintegrablerigid1}} = 26 \, .
    \]
    Among them, the only quadruples $(a,b,c;n)$  corresponding to sets $\TM_3(a,b,c;n)$
    such that the general element has a rational first integral are $(1,3,7;10)$ and $(4,6,7;25)$. They both
    correspond to the same irreducible and generically reduced component of $\Fol{3}$.  The eight remaining cases are:
    \begin{align*}
        & (1,2,5;7),(1,3,4;7),(1,3,4;13),(1,3,5;8), \\ & (1,3,5;11),(2,4,5;14),(2,4,5;17),(3,4,5;18).
    \end{align*}
    They correspond to four different subsets of $\Fol{3}$. We proceed to analyze each of them.

    \subsubsection*{$\TM_3(1,2,5;7) = \TM_3(3,4,5;18)$}
    The set $\TM_3(1,2,5;7)$ consists in the Zariski closure of the orbit under $\Aut(\mathbb P^3)$ of the foliation $\F$ defined in affine coordinates by the $1$-form
    \begin{equation}\label{E:(1,2,5;7)}
        \omega = (5zdy - 2y dz) + x( 5zdx - xdz) + y^2(2ydx - xdy) \, .
    \end{equation}

    If $g(x,y,z)= x^2 + 2y$ and $f(x,y,z)=x^5 + 5x^3y + (15/2)x y^2  - (15/2)z$ then $f^2/g^5$ is a first integral
    for $\F$.  It is clear that $\F$ lies in the irreducible component described by Proposition
    \ref{P:special log(2,5)} and Lemma \ref{L:conta25}. Thus the set $\TM_3(1,2,5;7) = \TM_3(3,4,5;18)$ is a proper closed subvariety of the irreducible component $\SLog(2,5)$.

    \subsubsection*{$\TM_3(1,3,5;11) = \TM_3(2,4,5;14)$}
    The set $\TM_3(2,4,5;14)$ consists in the Zariski closure of the orbit under $\Aut(\mathbb P^3)$ of the foliation $\F$ defined in affine
    coordinates by the $1$-form
    \[
        \omega = z ( 5zdy - 4ydz)  + xz ( 5zdx - 2x dz)+ y^2( 2ydx - xdy)
    \]
    Notice that $\omega$ is the pull-back under the morphism $\varphi(x,y,z) = (x,y,z^2)$ of the $1$-form
    presented in Equation (\ref{E:(1,2,5;7)}). If $g(x,y,z)= x^2 + 2y$ and $f(x,y,z)=x^5 + 5x^3y + (15/2)x y^2 - (15/2)z^2$ then $f^2/g^5$ is a first integral for $\F$. It follows that the set $\TM_3(1,3,5;11) = \TM_3(2,4,5;14)$
    is also a proper closed subvariety of $\SLog(2,5)$.

    \subsubsection*{$\TM_3(1,3,4;7) = \TM_3(1,3,4;13)$}
    The set $\TM_3(1,3,4;7)$ consists in the Zariski closure of the orbit under $\Aut(\mathbb P^3)$ of the foliation $\F$ defined in affine coordinates by the $1$-form
    \[
        \omega = (4zdy - 3y dz)  + y(3ydx - xdy) + x^2(4zdx - xdz) \, .
    \]
    We can rewrite $\omega$ as a complex multiple of $3f dg  - 4 g df$, where
    \begin{align*}
        f(x,y,z) &= x^3 + 3y \\
        g(x,y,z) & = x^4 + 4xy - 4z .
    \end{align*}
    If we set $(x_0,x_1,x_2,x_3,x_4) = ( 2 z, y,0,x,1)$ in Equation (\ref{E:CeDe}) then we recover a rational multiple of the $1$-form above.
        It follows that $\TM_3(1,3,4;7) = \TM_3(1,3,4;13)$ is a proper closed subvariety of $\SLog(3,4)$.

    \subsubsection*{$\TM_3(1,3,5;8) = \TM_3(2,4,5;17)$}
    The set $\TM_3(1,3,5;8)$ consists in the Zariski closure of the orbit under $\Aut(\mathbb P^3)$ of the foliation $\F$ defined in affine
    coordinates by the $1$-form
    \[
        \omega = (5zdy - 3y dz) + xy( 3ydx - xdy) + x^2(5zdx - xdz) \, .
    \]

    The irreducible hypersurfaces  cut out by the polynomials
    \begin{align*}
        f(x,y,z) &= x^3 + 3y \\
        g(x,y,z) &= x^5 + 5x^2y -10z
    \end{align*}
    are both invariant by $\F$ and the rational function $f^5/g^3$ is a rational first integral for $\F$.

    To conclude the proof of Proposition \ref{P:rationalfirstintegral} it remains to verify that $[\omega]$ is not contained
    in any of the irreducible components presented in Tables \ref{Tab:LOG}, \ref{Tab:split}, \ref{Tab:nonintegrablenonrigid}, and \ref{Tab:nonintegrablerigid}.

    Since $[\omega]$ does not have a polynomial integrating factor, Proposition \ref{P:belong to log} implies
    that $[\omega]$ does not belong to any of the logarithmic components $\Log(*)$. The general leaf of the foliation of $\F$ has degree
    $15$, while the general leaf of the general element of $\SLog(2,5)$ has degree $10$ and of $\SLog(3,4)$ has degree $12$. Since
    the set of foliations having an algebraic leaf of degree smaller than any given $k$ is closed, it follows that $[\omega]$ also does not belong
    to $\SLog(2,5)$ or to $\SLog(3,4)$. Thus $[\omega]$ does not belong to any of the irreducible components listed in Table \ref{Tab:LOG}.

    To verify that $[\omega]$ does not belong to any other irreducible component of the form $\TM_3(a,b,c;n)$, observe that
    the set
    \[
       \fix(\F) =  \{ v \in H^0(\mathbb P^3, T_{\mathbb P^3}) \, | \, i_v \omega =0 \}
    \]
    is one-dimensional and generated by the semi-simple vector field represented by $v_{(1,3,5)}$, and also represented by $v_{(1,3,5)}  -5R$
    which is conjugated to $-v_{(2,4,5)}$. At the same time, for any  foliation $\G$ in $\TM_3(a,b,c;n)$, the vector space $\fix(\G)$  has dimension  at least one, and
    if the dimension is one it must be generated by a vector field conjugated
    to $v_{(a,b,c)}$ or by a non semi-simple vector field. Therefore we can assume that
    the vector field represented by $v_{(a,b,c)}$ is conjugated to the vector field represented by $v_{(1,3,5)}$. Since there is no set in the Tables  \ref{Tab:split}, \ref{Tab:nonintegrablenonrigid}, and \ref{Tab:nonintegrablerigid} of the form $\TM_3(1,3,5;*)$, the result follows.
\end{proof}

\begin{cor}\label{C:multiplicativenonintegrable}
    If $\TM_3(a,b,c;n)$ is an irreducible component of $\Fol{3}$ such that the general element
    does not admit a rational first integral then $\TM_3(a,b,c;n)$ figures in Table \ref{Tab:nonintegrablenonrigid}
    or in Table \ref{Tab:nonintegrablerigid}.
\end{cor}

\tikzset{
    table/.style={
        matrix of nodes,
        row sep=-\pgflinewidth,
        column sep=-\pgflinewidth,
        nodes={
            rectangle,
            draw=black,
            align=center
        },
        minimum height=1.5em,
        text depth=0.5ex,
        text height=2ex,
        nodes in empty cells,
        every even row/.style={
            nodes={fill=gray!20}
        },
        column 1/.style={
            nodes={text width=4.65cm,align=left,font=\small}
        },
        column 2/.style={
            nodes={text width=0.8cm,align=center},font=\small},
        column 3/.style={
            nodes={text width=0.8cm,align=center},font=\small},
        column 4/.style={
            nodes={text width=5.2cm,align=left, font=\small}
        },
        row 1/.style={
            nodes={ font=\bfseries
            }
        }
    }
}

\begin{table}
\begin{tikzpicture}
\matrix (mat) [table,text width=3cm]
{
    {\bf Set} & ${\dim}$     &  $\mathrm{Zdim}$      & Comment   \\
    $\TM_3(1,2,5;7) = \TM_3(3,4,5;18)$ & $14$  & $19$  & Contained in $\SLog(2,5)$. \\
    $\TM_3(1,3,4;7) = \TM_3(1,3,4;13)$ & $14$  & $21$  & Contained in $\SLog(3,4)$. \\
    $\TM_3(1,3,5;11)= \TM_3(2,4,5;14)$ & $14$  & $19$  & Contained in $\SLog(2,5)$. \\
    $\TM_3(1,3,7;10)=\TM_3(4,6,7;25)$ & $14$     & $14$  &  Appears in Table \ref{Tab:split}. \\
    $\TM_3(1,3,5;8) = \TM_3(2,4,5;17)$ & $14$  & $19$  &  Contained in a non-listed component.\\
};
\end{tikzpicture}
\caption{\small Candidates with rational first integrals. }
\end{table}

\section{Foliations on \texorpdfstring{$\mathbb P^3$}{P3}  tangent to additive actions}\label{S:additive}

In this section we study foliations on $\mathbb P^3$ tangent to algebraic actions
of $\mathbb C$ and show that they are all degenerations of foliations tangent to multiplicative
actions.

\subsection{Setup}
Let $\varphi:\C\times \Pj^3\rightarrow \Pj^3$ be an algebraic $\C$-action. For a
suitable choice of coordinates this action is determined by a nilpotent homogeneous vector field on $\C^4$ which has
the following form:
\[
    v_{(a,b)}=x_1\frac{\partial}{\partial x_0}+ax_2\frac{\partial}{\partial x_1}+bx_3\frac{\partial}{\partial x_2}\,,
\]
where $a,b\in\{0,1\}$. When both $a$ and $b$ are zero, the vector field $v_{(0,0)}$ has codimension one singularities
and foliations tangent to it are linear pull-backs. In what follows, we will focus on foliations
invariant by $\mathbb C$-actions which are not linear pull-backs.

\begin{lemma}
    If $\F$ is a foliation on $\Pj^3$ given by a homogeneous $1$-form $\omega$ tangent to $v_{(a,b)}$
    then  $L_{v_{(a,b)}}\omega=0$.
\end{lemma}
\begin{proof}
    Let $\psi_t: \C^4 \to \C^4$ be the flow determined by $v_{(a,b)}$. Since $(\mathbb C,+)$ has
    no non-trivial characters, it follows that $\psi_t^* \omega= \omega$. The result follows from the differentiation
    of this expression.
\end{proof}

For every $d\in \N$ and every $a,b\in\{0,1\}$, set
\[
A_d(1,a,b) = \left\{ \omega \in H^0(\Pj^3,\Omega_{\Pj^3}^1(d+2))\,|\,
\text{$i_{v_{(a,b)}} \omega=0$ and $L_{v_{(a,b)}}\omega=0$}  \right\}\,.
\]

Analogously to the definition of $\TM_d(a,b,c;n)$ in Section \ref{S:mult}, we define $\TA_d(1,a,b)$
as the closure of the image of the rational map
\begin{align*}
    \varphi_{(a,b)} : \Aut(\mathbb P^3) \times \mathbb P(A_d(1,a,b)) &\dashrightarrow \mathbb P H^0(\mathbb P^3, \Omega^1_{\mathbb P^3}(d+2)) \\
    (\varphi, [\omega] ) & \longmapsto \varphi^* [\omega] \, .
\end{align*}
intersected with $\Open[3]{d}$.

\begin{lemma}\label{L:dimTAd}
    If $d \ge 3$ and $(a,b) \in \{ (0,1),(1,0), (1,1) \}$  then the dimension of $\TA_d(1,a,b)$ is equal to
    \[
        \dim \Aut(\mathbb P^3) + \dim \mathbb P(A_d(1,a,b)) - \dim \{ v \in H^0(\mathbb P^3, T_{\mathbb P^3}) ; [v,v_{(a,b)}] \wedge v_{(a,b)} =0 \}  \, .
    \]
\end{lemma}
\begin{proof}
    It suffices to observe that $A_d(1,a,b)$ always has an element defining a foliation with codimension two
    singularities and apply the  argument used in the proof of Lemma \ref{L:dimTMd}.
\end{proof}

\subsection{The sets $\TA_d(1,1,0)$ and $\TA_d(1,0,1)$}
The sets $\TA_d(1,1,0)$ and $\TA_d(1,0,1)$ admit a simple uniform description, for arbitrary $d\ge 2$, as the lemma below shows.

\begin{lemma}\label{L:TA110}
    For $d\ge 2$, the sets $\TA_d(1,1,0)$ and $\TA_d(1,0,1)$ are strictly
    contained in $\TM_d(1,1,2;d+2)$. In particular, they are not irreducible components
    of $\Fol{d}$.
\end{lemma}
\begin{proof}
    Let $L_1, L_2 \in H^0(\mathbb P^3,\mathcal O_{\mathbb P^3}(1))$ be two linear forms and
    $Q \in H^0(\mathbb P^3,\mathcal O_{\mathbb P^3}(2))$ be a quadratic form.
    If the rational map
    \begin{align*}
        \varphi : \mathbb P^3 & \dashrightarrow \mathbb P(1,1,2) \\
        (x_0: x_1 : x_2 :x_3) &\mapsto (L_1:L_2:Q)
    \end{align*}
    is dominant then its fibers define a one-dimensional foliation $\mathcal G$ on $\mathbb P^3$ of degree at most one.
    If the degree is exactly one then $\G$ is defined by a global vector field. For example, if $L_1 = x_0$, $L_2=x_1$,
    and $Q=x_2 \cdot x_3 $ then $\G$ is defined by the vector field
    \[
        v =  x_0 \frac{\partial }{\partial x_0} + x_1 \frac{\partial }{\partial x_1} + 2 x_2 \frac{\partial }{\partial x_2} \, .
    \]
    Small perturbations of $L_1, L_2$, and $Q$ will also define foliations tangent to semi-simple vector fields conjugated to $v$. The reader can verify that the precise condition to guarantee that a semi-simple vector field defines $\G$ is that the line $\{L_1=L_2=0\}$ intersects
    $\{Q=0\}$  at two distinct points. Under these conditions,
    the set $\Sigma = \{ dL_1 \wedge dL_2 \wedge dQ =0 \}$ has codimension at least two. Outside $\Sigma$, the differential of $\varphi$  has maximal rank.
    Therefore the set $\TM_d(1,1,2;d+2)$ can be described as the closure of pull-backs
    of foliations on $\mathbb P(1,1,2)$ with normal sheaf equal to $\mathcal O_{\mathbb P(1,1,2)}(d+2)$ under
    rational maps of the same form as $\varphi$.

    If $L_1 =x_2$, $L_2=x_3$, and    $Q =x_1^2 - 2 x_0 x_2$ then $\G$ is defined by
    \[
        v = x_1 \frac{\partial }{\partial x_0} +  x_2 \frac{\partial }{\partial x_1} \, .
    \]
    If instead $L_1 = x_1$, $L_2=x_3$, and $Q =  x_1x_2 - x_0 x_3$ then $\G$ is defined by
    \[
        v =  x_1 \frac{\partial }{\partial x_0} +  x_3 \frac{\partial }{\partial x_2} \, .
    \]
    In both cases, the corresponding $\varphi$ does not contract hypersurfaces showing that the sets
    $\TA_d(1,0,1)$ and $\TA_d(1,1,0)$ are both contained in $\TM_d(1,1,2;d+2)$.
\end{proof}

\begin{remark}
    The dimension of $\TA_3(1,1,0)$ is $20$ and therefore it is a hypersurface of $\TM_3(1,1,2;5)$.
    A computation shows that the Zariski tangent space of $\Fol{3}$ at a general point of $\TA_3(1,1,0)$ has dimension
    $21$. Therefore, a general point of $\TA_3(1,1,0)$ is contained in the smooth locus of $\TM_3(1,1,2;5)$. We also point
    out that the general element of $\TA_3(1,1,0)$ is a foliation with finitely many non-Kupka singularities.

    The dimension of $\TA_3(1,0,1)$ is $18$ and the Zariski tangent space of $\Fol{3}$ at a general point of it has dimension $25$.
    The general foliation in $\TA_3(1,0,1)$ has infinitely many non-Kupka singularities.
\end{remark}

\subsection{The set $\TA_3(1,1,1)$}
We now turn our attention to the set $\TA_3(1,1,1)$.

\begin{lemma}\label{L:TA111}
    The set $\TA_3(1,1,1)$ is strictly contained in $\TM_3(1,2,3;7)$. In particular, $\TA_3(1,1,1)$ is  not an irreducible component
    of $\Fol{3}$.
\end{lemma}
\begin{proof}
    Consider the family of homogeneous vector fields
    \[
        v_{\epsilon} = v_{(1,1)} + \epsilon v_{(1,2,3)}
    \]
    parameterized by $\epsilon \in \mathbb C$.  Consider also the family of linear maps
    \begin{align*}
         \beta_{\epsilon} : H^0(\mathbb P^3, \Omega^1_{\mathbb P^3}(5)) & \longrightarrow H^0(\mathbb P^3,\mathcal O_{\mathbb P^3}(5)) \times H^0(\mathbb P^3, \Omega^1_{\mathbb P^3}(5)) \\
        \omega & \mapsto \left( i_{v_{\epsilon}} \omega  , L_{v_{\epsilon}} \omega - 7 \epsilon \omega \right) \, .
    \end{align*}

    For $\epsilon \neq 0$, since $v_{\epsilon}$ is conjugated to $\epsilon v_{(1,2,3)}$,
    the kernel of $\beta_{\epsilon}$ is conjugated to $V_3(1,2,3;7)$. For $\epsilon =0$, the kernel $\beta_0$ coincides with
    $A_3(1,1,1)$.

    Explicitly computing the kernel of $\beta_{\epsilon}$, we verified that every element in the kernel of $\beta_0$ is
    a limit of elements in the kernel of $\beta_{\epsilon}$ for $\epsilon \neq 0$. Alternatively, this fact also follows
    from the observation that $\dim \ker \beta_{\epsilon}$ does not depend on $\epsilon$.
            In any case, we obtain
    that $\TA_3(1,1,1)$ is strictly contained in $\TM_3(1,2,3;7)$.
\end{proof}

\begin{remark}\label{R:generalizando}
    Following a similar approach we verified that $\TA_d(1,1,1)$ is contained in $\TM_d(1,2,3;2d+1)$ also for $d=2,4,5,6$. It seems
    natural to expect that $\TA_d(1,1,1)$ is contained in $\TM_d(1,2,3;2d+1)$ for every $d \ge 2$.
    When $d=2$, the sets $\TA_2(1,1,1)$ and $\TM_2(1,2,3;5)$ are equal and coincide with the so
    called exceptional component of $\Fol{2}$.
\end{remark}

\subsection{Synthesis}Lemmas \ref{L:TA110} and \ref{L:TA111} above immediately imply the main result of this section.

\begin{prop}\label{P:additive}
    The only irreducible component of $\Fol{3}$ for which its general element corresponds to
    a foliation tangent to an algebraic action of $(\mathbb C,+)$ is the linear pull-back component $\LPB(3)=\TA_3(1,0,0)$.
\end{prop}

If the expectation presented in  Remark \ref{R:generalizando} is confirmed, then the result above holds for every the irreducible
components of $\Fol{d}$, for every $d\ge 3$. As already pointed out, the general element of exceptional component of $\Fol{2}$ is tangent
to a an algebraic action of $(\mathbb C,+)$.

\section{Proofs of the main results}\label{S:proofs}

\subsection{Proof of Theorem \ref{T:structure}}\label{SS:proof of structure}
Since the argument is a simple extension of the ones carried out in the proof of \cite[Theorem A]{LorayPereiraTouzet17},
we will use the same notations of that paper.

Let $\F$ be a foliation of degree three on $\mathbb P^n$, $n\ge 3$. If $T_{\mathcal F}$ is the tangent sheaf of $\mathcal F$
then its determinant $\det T_{\mathcal F}$ is isomorphic to $\mathcal O_{\mathbb P^n}(n-4)= \mathcal O_{\mathbb P^n}(\dim \mathcal F - \deg \mathcal F )$.

Let $f: \mathbb P^1 \to \mathbb P^n$ be the parameterization of a line $\ell$ on $\mathbb P^n$. Assume $f$ is generically transverse to
$\mathcal F$. Consider the Birkhoff-Grothendieck decomposition of $f^* T_{\mathcal F}$ in a direct sum of line-bundles. Because
$f$ is generically transverse to $\F$ and the normal bundle of $\ell$ is $\mathcal O_{\ell}(1)^{\oplus n-1}$, every summand of the decomposition has degree at most one. Therefore, we have the following possibilities
for $f^* T_{\mathcal F}$:
\begin{enumerate}
    \item $f^* T_{\mathcal F} \simeq \mathcal O_{\mathbb P^1}(-2) \oplus \mathcal O_{\mathbb P^1}(1)^{\oplus n-2}$; or
    \item $f^* T_{\mathcal F} \simeq \mathcal O_{\mathbb P^1}(-1) \oplus \mathcal O_{\mathbb P^1} \oplus \mathcal O_{\mathbb P^n}(1)^{\oplus n-3}$; or
    \item $f^* T_{\mathcal F} \simeq \mathcal O_{\mathbb P^1}^{\oplus 3} \oplus \mathcal O_{\mathbb P^1}(1)^{\oplus n-4}$;
\end{enumerate}

If $\F$ does not admit a rational first integral and is not defined by closed rational $1$-form without codimension one zeros
then, as argued in the proof of \cite[Theorem A]{LorayPereiraTouzet17}, there exists an algebraically integrable foliation
$\mathcal A$ such that $T_{\mathcal A}$ is a subsheaf of $T_{\mathcal F}$ (i.e., $\mathcal A$ is contained in $\mathcal F$), and
$h^0(\mathbb P^1, f^* T_{\mathcal A}) = h^0(\mathbb P^1, f^* T_{\mathcal F})$. It follows that
\begin{enumerate}
    \item $f^* T_{\mathcal A} \simeq  \mathcal O_{\mathbb P^1}(1)^{\oplus n-2}$; or
    \item $f^* T_{\mathcal A} \simeq  \mathcal O_{\mathbb P^1} \oplus \mathcal O_{\mathbb P^1}(1)^{\oplus n-3}$.
\end{enumerate}
Note that the possibility $f^* T_{\mathcal F} \simeq \mathcal O_{\mathbb P^1}^{\oplus 3} \oplus \mathcal O_{\mathbb P^1}(1)^{\oplus n-4}$ is excluded because in that case $f^* T_{\mathcal A}$ would be isomorphic to $f^*T_{\mathcal F}$ and, consequently,
$\F$ would be algebraically integrable.

To conclude, observe that $\det T_{\mathcal A} \simeq \mathcal O_{\mathbb P^n}(\dim \mathcal A - \deg \mathcal A)$ implies $\dim \mathcal A = n-2$ and $\deg \mathcal A \in \{ 0 ,1\}$. If $\deg \mathcal A=0$ then $\F$ is a linear pull-back of a foliation on $\mathbb P^2$. Otherwise, $\mathcal A$ is an algebraically integrable codimension two foliation of degree one tangent to $\mathcal F$.  \qed

\subsection{Proof of Theorem \ref{THM:A}}
The theorem below is a refined version of Theorem \ref{THM:A} from the Introduction.

\begin{thm}\label{T:more precise}
    The space of codimension one foliations of degree three on $\mathbb P^n$, $n\ge 3$,  has exactly
    $18$ distinct irreducible components whose general elements correspond to foliations which
    do not admit non-constant rational first integrals. They are
    \begin{enumerate}
        \item one irreducible component parameterizing linear pull-backs of degree three foliations on $\mathbb P^2$: $\LPB(3)$;
        \item four irreducible components of logarithmic type:
        \[
            \Log(1,1,1,1,1), \Log(1,1,1,2), \Log (1,1,3), \Log(1,2,2) ;
        \]
        \item one irreducible component parameterizing pull-backs of foliations on $\mathbb P=\mathbb P(1,1,2)$
        with normal sheaf $\mathcal O_{\mathbb P}(5)$ under rational maps of the form
        \[
            (x_0: \ldots: x_n) \mapsto (L_1:L_2:Q)
        \]
        where $L_1, L_2 \in \C[x_0, \ldots, x_n]$ are linear forms, and $Q \in \C[x_0, \ldots, x_n]$ is a quadratic form.
        \item\label{I:13} twelve  irreducible components parameterizing foliations which are linear pull-backs of foliations
        on $\mathbb P^3$ tangent to an algebraic action of $(\mathbb C^*,\cdot)$ with fixed point set of codimension at least two of the form $\TM_3(a,b,c;n)$, $(a,b,c;n)\neq (1,1,2;5)$,  listed in Tables \ref{Tab:nonintegrablenonrigid} and \ref{Tab:nonintegrablerigid}.
    \end{enumerate}
\end{thm}
\begin{proof}
    Let $\F$ be a degree three foliation on $\mathbb P^n$, $n\ge 3$, without a rational first integral. Assume also
    that $\F$ is a sufficiently general element of an irreducible component of $\Fol[n]{3}$.
    According to Theorem \ref{T:structure},
    \begin{enumerate}
        \item\label{I:Final LPB} $\F$ is a linear pull-back of degree three foliation on $\mathbb P^2$; or
        \item\label{I:Final Log} $\F$ is defined by a closed rational $1$-form without codimension one zeros; or
        \item\label{I:Final TM} $\F$ is tangent to an algebraically integrable codimension two foliation $\mathcal A$ of degree one.
    \end{enumerate}

    The set of foliations in $\Fol[n]{3}$ which are linear pull-backs from $\mathbb P^2$ is an irreducible
    component of $\Fol[n]{3}$, see for instance \cite[Corollary 5.1]{CukiermanPereira}.

    If $\F$ is defined by a  closed rational $1$-forms without codimension one zeros, then
    it admits a polynomial integrating factor.
    Proposition \ref{P:belong to log} implies that $\F$ belongs to one of the logarithmic components.
    The logarithmic components are irreducible and generically reduced
    components of $\Fol[n]{3}$, see Theorem \ref{T:log classifica}. They are in bijection with the partitions of $5$ with at
    least two positive integral summands. Moreover, a general element of $\Log(d_1, \ldots, d_k)$ is an algebraically integrable
    foliation  if, and only if, $k \le 2$. It follows that $\F$ belongs to  $\Log(1,1,1,1,1), \Log(1,1,1,2), \Log (1,1,3)$, or
    $\Log(1,2,2)$.

    If $\F$ is not a linear pull-back from $\mathbb P^2$ and does not admit a polynomial integrating factor, then there exists a
    codimension two foliation $\mathcal A$ of degree one tangent to $\mathcal F$. According to \cite[Theorem 6.2]{LorayPereiraTouzet13},
    the foliation $\mathcal A$ is determined by the fibers of a dominant rational map $\mathbb P^n \dashrightarrow \mathbb P(1,1,2)$ with irreducible general fiber and defined by two linear forms and one quadratic form; or is a linear pull-back of a foliation on $\mathbb P^3$ defined by a global holomorphic vector field.

    Assume first that $\mathcal A$ is defined by the fibers of
    \begin{align*}
        \varphi: \mathbb P^n & \dashrightarrow \mathbb P(1,1,2) \\
        (x_0: \ldots: x_n) &\mapsto ( L_1 : L_2 : Q ) \, .
    \end{align*}
    Therefore $\mathcal A$ is defined by the $2$-form $\omega \in H^0(\mathbb P^n, \Omega^2_{\mathbb P^n}(4))$ given in homogeneous coordinates by
    \[
        \omega = i_R (dL_1 \wedge dL_2 \wedge dQ) \, .
    \]
    Notice that the vanishing locus of $\omega$ coincides with the closure of the locus where the morphism $\varphi_{|\mathbb P^n - \{ L_1 = L_2 = Q=0\}}$ has rank smaller than two. Because $\omega$ defines a foliation of degree one, the singular locus of $\omega$ has codimension at least two. Therefore, $\varphi$  is a submersion outside a set of codimension at least two.
    Consequently, for any foliation $\mathcal G$ on $\mathbb P(1,1,2)$, the normal bundle of $\varphi^* \mathcal G$ is equal to $\varphi^* N_{\mathcal G}$. Since $\F$ is tangent to $\mathcal A$ and $N_{\mathcal F} = \mathcal O_{\mathbb P^n}(5)$, there exists a foliation $\G$ on    $\mathbb P(1,1,2)$ with $N_{\mathcal G} = \mathcal O_{\mathbb P}(5)$ such that $\F = \varphi^* \mathcal \G$.

    Assume now that $\mathcal A$ is a linear pull-back of a degree one foliation by curves on $\mathbb P^3$. In this case, the foliation $\F$ also
    is a linear pull-back of a foliation $\G$ on $\mathbb P^3$. Proposition \ref{P:Key argument} implies that $\G$ is tangent to the orbits of a one-dimensional algebraic group $G$. Since we are assuming that $\F$ is general in the irreducible component of $\Fol[n]{3}$ containing it, Proposition $\ref{P:additive}$ implies that $G$ is the multiplicative group $(\mathbb C^*,\cdot)$.  Corollary \ref{C:multiplicativenonintegrable} implies that $\G$ belongs to one of the thirteen sets listed in Tables \ref{Tab:nonintegrablenonrigid} and \ref{Tab:nonintegrablerigid}.
    If $\G$ belongs to $\TM_3(1,1,2;5)$ then $\F$ fits into the description of the previous paragraph. If instead $\G$ does not belong to
    $\TM_3(1,1,2;5)$, any sufficiently small
     deformation of $\F$ is also a linear pull-back from $\mathbb P^3$.
\end{proof}

\subsection{Theorem \ref{THM:B}}
Arguments similar to the ones used to proof Theorem \ref{T:more precise} give the following more precise version of Theorem \ref{THM:B}.

\begin{thm}\label{T:AlgInt}
    The space of codimension one foliations of degree three on $\mathbb P^n$, $n\ge 3$, has at least
    $6$ distinct irreducible components whose general elements correspond to algebraically integrable foliations.
    In dimension $n=3$, they are
    \begin{enumerate}
        \item two irreducible components of logarithmic type: $\Log(1,4), \Log(2,3)$;
        \item two irreducible components of special logaritmic type: $\SLog(2,5), \SLog(3,4)$;
        \item one irreducible component parameterizing foliations tangent to a multiplicative action of $\mathbb C^*$ with
        fixed point set of codimension at least two: $\TM_3(1,3,7;10)$;
        \item at least one irreducible component containing the set $\TM_3(1,3,5;8)$.
    \end{enumerate}
    In dimension $n \ge 4$, for each one of the $6$ irreducible components above, there exists
    at least one irreducible component of $\Fol[n]{3}$ such that the restriction of a general
    element to a general $\mathbb P^3 \subset \mathbb P^n$ is a general element of the respective
    irreducible component of $\Fol[3]{3}$.
\end{thm}

\subsection{Some open problems}

To conclude we present some problems that came to us in the course of our investigations.

\subsubsection{Irreducible components}
From our results, we already know that $\Fol[n]{3}$ has at least $24$ distinct irreducible components.
Is this the right number of components, or are there irreducible components not considered in this paper?

\begin{problem}\label{Prob:irreducible}
     How many irreducible components does $\Fol[n]{3}$ have?
\end{problem}

Of course,  the generic elements of the missing irreducible components correspond to foliations
with rational first integrals.

\begin{problem}\label{Prob:irreducible bis}
    Does the number of irreducible components of $\Fol[n]{3}$ vary with $n$?
\end{problem}

Notice that we have not excluded the possibility in the statement Theorem \ref{THM:B},
of the existence of two or more irreducible components of $\Fol[n]{3}$, $n \ge 4$, whose
restriction of a general element to a general $\mathbb P^3 \subset \mathbb P^n$ lies
in the same component of $\Fol[3]{3}$. Of course, Theorem \ref{T:log classifica} guarantees that
this does not happen with $\Log(1,4)$ and $\Log(2,3)$.

\subsubsection{Algebraically integrable foliations}
We have the impression that we still lack ideas and tools to deal with irreducible components of $\Fol[n]{3}$ (or $\Fol[n]{d}$) for which
the general member is algebraically integrable. One problem on the subject that seems interesting is the following.

\begin{problem}\label{Prob:Poincare}
    Let $\Sigma \subset \Fol[n]{d}$ be an irreducible component for which the general element corresponds to
    a foliation $\F$ with a rational first integral.
    \begin{enumerate}
        \item\label{I:Poincare1} Give a bound for the degree of the general leaf of $\F$ in function of $d$ and $n$.
        \item\label{I:Poincare2} Let $f: \mathbb P^n \dashrightarrow \mathbb P^1$ be a rational first integral for $\F$ with irreducible general
        fiber. Give a bound for the maximal number of non-reduced fibers of $f$  in function of $d$ and $n$.
    \end{enumerate}
\end{problem}

Item (\ref{I:Poincare1}) of Problem \ref{Prob:Poincare} is a variant of the so-called Poincaré Problem for foliations, see the Introduction of \cite{MR3957403} for a brief history of this problem. Contrary to the classical problem, which does not admit an answer only in function of the degree and the analytical type of the singularities of the foliation, see \cite{MR1914932}; the existence of bounds as requested by Problem \ref{Prob:Poincare} follows from the fact that  $\Fol[n]{d}$ has finitely many irreducible components.

From the list of irreducible components presented in this paper, we see that for foliations of degree three the answer to Item (\ref{I:Poincare1}) is at least $21$ (realized by $\TM_3(1,3,7;10))$) and the answer to Item (\ref{I:Poincare2}) is at least $3$, realized by any of the components present in Theorem  \ref{T:AlgInt} with the exception of $\Log(1,4)$ and $\Log(2,3)$. A sharp answer to Problem \ref{Prob:Poincare} for foliations of degree $3$ would be an important step toward an answer to Problem \ref{Prob:irreducible}.

\subsubsection{A conjectural necessary condition for the smoothness of $\Fol[n]{d}$}

We finish things off by  putting forward a conjecture. Although it is not strictly
related to the classification of irreducible components of $\Fol{3}$, it was
suggested by the computations we carried out while preparing this paper.

\begin{conj}\label{Conj:smooth}
    Let $\F$ be a degree $d$ codimension one foliation on $\mathbb P^n$.
    If  the set of non-Kupka singularities has codimension at least three, then
    $\F$ corresponds to a smooth point of $\Fol{d}$, i.e., there exists
    a unique irreducible component of $\Fol[n]{d}$ containing $\F$ and
    its dimension coincides with the dimension of the Zariski tangent
    space of $\Fol[n]{d}$ at $\mathcal F$.
\end{conj}

The proofs that logarithmic components are generically reduced in \cite{CukiermanPereiraVainsencher} (divisor of
poles with two components) and \cite{CukiermanAceaMassri} (general case) provide further evidence for Conjecture \ref{Conj:smooth}.

\end{document}